\documentclass[11pt, oneside, reqno]{amsart}
\usepackage[top=2.5cm, bottom=2.5cm ,left=3cm, right=3cm]{geometry}

\usepackage{microtype}
\usepackage[
    backend=biber,
    style=alphabetic,
    backref=true,
    url=false,
    doi=false,
    isbn=false
]{biblatex}
\usepackage{mathtools}
\usepackage[regular]{newcomputermodern}
\usepackage{thmtools}
\usepackage{tikz-cd}
\usepackage{enumitem}
\setlist{itemsep=1pt, topsep=4pt}
\usepackage{graphicx}
\usepackage{comment}

\usepackage{xcolor} 
\definecolor{citeblue}{rgb}{0,0.35,0.75}
\usepackage[
    colorlinks=true,
    linkcolor=citeblue,
    citecolor=citeblue,
    urlcolor=citeblue
]{hyperref}
\usepackage[nameinlink,noabbrev]{cleveref}

\numberwithin{equation}{section}

\crefname{equation}{}{}
\Crefname{section}{\S\!}{}

\theoremstyle{plain}
\newtheorem{theorem}{Theorem}[section]
\newtheorem{lemma}[theorem]{Lemma}
\newtheorem{proposition}[theorem]{Proposition}
\newtheorem{corollary}[theorem]{Corollary}
\newtheorem{conjecture}[theorem]{Conjecture}
\newtheorem{maintheorem}{Theorem}

\theoremstyle{definition}
\newtheorem{definition}[theorem]{Definition}
\newtheorem{example}[theorem]{Example}
\newtheorem{remark}[theorem]{Remark}

\newenvironment{flexible}[1]
  {\par\medskip\noindent\textbf{#1.\,}\normalfont}
  {\par\ignorespacesafterend}

\newcounter{proofstep}
\newcommand{\proofstep}[1][]{%
  \par\medskip
  \refstepcounter{proofstep}%
  \noindent\textbf{Step \theproofstep%
    \if\relax\detokenize{#1}\relax
      .
    \else
      .\ #1.
    \fi
  }
}

\newcounter{claimnum}
\newcommand{\claim}[0]{%
  \par\smallskip
  \refstepcounter{claimnum}%
  \noindent\textbf{Claim \theclaimnum. 
  }
}

\newcommand{\chow}[1]{\pmb{\lbrace}  #1  \pmb{\rbrace}}

\newcommand{\cO}{\mathcal{O}}

\newcommand{\cD}{\mathcal{D}}
\newcommand{\cE}{\mathcal{E}}

\newcommand{\cJ}{\mathcal{J}}

\newcommand{\cP}{\mathcal{P}}

\newcommand{\bL}{\mathbf{L}}

\newcommand{\bR}{\mathbf{R}}
\newcommand{\bS}{\mathbf{S}}

\newcommand{\sM}{\mathscr{M}}
\newcommand{\sP}{\mathscr{P}}

\newcommand{\rH}{\mathrm{H}}

\newcommand{\rK}{\mathrm{K}}

\newcommand{\rT}{\mathrm{T}}

\newcommand{\rU}{\mathrm{U}}

\newcommand{\CC}{\mathbb{C}}
\newcommand{\GG}{\mathbb{G}}

\newcommand{\RR}{\mathbb{R}}
\newcommand{\UU}{\mathbb{U}}
\newcommand{\QQ}{\mathbb{Q}}
\newcommand{\ZZ}{\mathbb{Z}}
\newcommand{\DD}{\mathbb{D}}

\newcommand{\bfv}{\mathbf{v}}

\newcommand{\fro}{\mathfrak{o}}
\DeclareMathOperator{\Db}{\mathrm{D}^\mathrm{b}}

\DeclareMathOperator{\CH}{CH}
\DeclareMathOperator{\ch}{ch}

\DeclareMathOperator{\tr}{tr}
\DeclareMathOperator{\id}{id}
\DeclareMathOperator{\NS}{NS}
\DeclareMathOperator{\Br}{Br}

\DeclareMathOperator{\Hom}{Hom}
\DeclareMathOperator{\Ext}{Ext}
\DeclareMathOperator{\Aut}{Aut}
\DeclareMathOperator{\Bir}{Bir}

\DeclareMathOperator{\Alb}{Alb}
\DeclareMathOperator{\alb}{alb}
\DeclareMathOperator{\Pic}{Pic}

\DeclareMathOperator{\Stab}{Stab}
\DeclareMathOperator{\Mon}{Mon}

\DeclareMathOperator{\even}{even}
\DeclareMathOperator{\odd}{odd}

\DeclareMathOperator{\Km}{Km}
\DeclareMathOperator{\km}{Km}
\DeclareMathOperator{\Kum}{Kum}

\DeclareMathOperator{\rank}{rank}
\DeclareMathOperator{\Coh}{Coh}

\DeclareMathOperator{\GL}{GL}
\DeclareMathOperator{\SO}{SO}

\DeclareMathOperator{\et}{\text{\'et}}

\newcommand{\cf}{\textit{cf.~}}

\title[Bloch's conjecture for twisted abelian surfaces]{Bloch's conjecture for equivalences between twisted abelian surfaces and applications}
\date{\today}

\author{Zaiyuan Chen}
\address{Zaiyuan Chen, Shanghai Center for Mathematical Sciences, Fudan University, Jiangwan Campus, Shanghai, 200438, China}
\email{zaiyuanchen16@fudan.edu.cn}

\author{Zhiyuan Li}
\address{Zhiyuan Li, Shanghai Center for Mathematical Sciences, Fudan University, Jiangwan Campus, Shanghai, 200438, China}
\email{zhiyuan\_li@fudan.edu.cn}

\author{Ruxuan Zhang}
\address{Ruxuan Zhang, School of Mathematical Sciences, Fudan University, Handan Campus, Shanghai, 200433, China}
\email{rxzhang18@fudan.edu.cn}

\begin{document}

\begin{abstract}

The Beauville--Voisin conjecture predicts a canonical descending filtration on the Chow group of zero-cycles of a hyperk\"{a}hler variety, opposite to the conjectural Bloch--Beilinson filtration. A basic test for such filtrations is a Bloch-type principle: the action on zero-cycles should be governed by the action on the holomorphic symplectic form. While this principle has been verified in several cases of hyperk\"ahler varieties of $\mathrm{K3}^{[n]}$-type, the $\mathrm{Kum}_n$-type case remains much less understood.

In this paper, we study this problem through twisted abelian surfaces and their associated $\mathrm{Kum}_n$-type varieties. We first construct a natural action of autoequivalences of twisted abelian surfaces on the Albanese kernel and prove Bloch’s conjecture for all (anti-)symplectic autoequivalences. As an application, we prove the corresponding Bloch conjecture for symplectic birational automorphisms of twisted modular $\mathrm{Kum}_n$-type varieties; in particular, this applies to those admitting a birational Lagrangian fibration. 

Finally, we introduce and study a Shen--Yin--Zhao type filtration on twisted modular varieties and compare it with Voisin’s filtration in the sixfold case. We also establish the anti-symplectic Bloch conjecture for twisted modular $\mathrm{Kum}_3$-type varieties.

\end{abstract}

\maketitle
\setcounter{tocdepth}{1}

\section{Introduction}

\subsection{Bloch's conjecture for derived equivalences}

The Chow group of zero-cycles on a smooth projective surface $X$ has two elementary invariants: degree and Albanese class. After these are fixed, the remaining part is the Albanese kernel
\[
\CH_0(X)_{\alb}:=\ker\bigl(\CH_0(X)_{\hom}\to \Alb(X)\bigr),
\]
where the transcendental geometry of the surface is expected to appear. Bloch's principle predicts that algebraic correspondences should act on this residual group through their action on holomorphic forms \cite{BlochKasLieberman}.

In the derived setting, Fourier--Mukai kernels provide a  natural and rich supply of such correspondences. For K3 and abelian surfaces---and more generally in the twisted setting---one is led to ask whether the action on $\CH_0(X)_{\alb}$ is determined by the action on the holomorphic two-form.

\begin{conjecture}[Bloch's conjecture for (anti‑)equivalences]\label{conj:autoeq}
Let $(X,\alpha)$ and $(X',\alpha')$ be twisted projective K3 surfaces or abelian surfaces. For two derived (anti‑)equivalences
$\Psi_1,\Psi_2 : \mathrm{D}^{\mathrm{b}}(X,\alpha) \to \mathrm{D}^{\mathrm{b}}(X',\alpha')$,
there are natural induced maps
\[
\Psi_i^{\mathrm{CH}_0}: \mathrm{CH}_0(X)_{\mathrm{alb}} \longrightarrow \mathrm{CH}_0(X')_{\mathrm{alb}},\qquad
\Psi_i^{2,0}: \rH^0(X,\Omega_X^2) \longrightarrow \rH^0(X',\Omega_{X'}^2).
\]
Then
\[
\Psi_1^{2,0} = \Psi_2^{2,0} \quad\Longleftrightarrow\quad 
\Psi_1^{\mathrm{CH}_0} = \Psi_2^{\mathrm{CH}_0}.
\]
\end{conjecture}

The Chow-theoretic map $\Psi^{\CH_0}$ is not immediately obvious and requires elaboration. For a derived (anti-)equivalence $\Psi: \Db(X,\alpha) \to \Db(X',\alpha')$ with kernel $\mathcal{E} \in \Db(X \times X', \alpha^{-1} \boxtimes \alpha')$, the action on Chow groups is defined by the correspondence
\[
\Psi^{\CH_*}(\gamma) := \pi_{X'*}\bigl( \ch(\mathcal{E}) \cdot \pi_X^*(\gamma) \bigr),
\]
where $\pi_X, \pi_{X'}$ are the projections. For untwisted K3 surfaces, Huybrechts \cite[Theorem 2]{Huybrechts10} proved that $\Psi^{\CH^*}$ preserves the Beauville--Voisin ring; this gives a well-defined map on the Albanese kernels:
\[
\Psi^{\CH_0} : \CH_0(X)_{\mathrm{alb}} \longrightarrow \CH_0(X')_{\mathrm{alb}}.
\]
For twisted surfaces, the same role is played by the twisted Chern character (see \cite[\S 5]{CLZZ:2024}; the construction is carried out in \Cref{sec:bloch-surface}). In parallel, the twisted Chern character $\ch(\cE)$ defines a Hodge isometry that restricts to a cohomological action $\Psi^{2,0}: \mathrm{H}^0(X,\Omega_X^2) \to \mathrm{H}^0(X',\Omega_{X'}^2)$ (see \Cref{subsection:FM-action}).

\begin{remark}\label{rmk:conj=auto}
This conjecture easily reduces to the autoequivalence case. By setting $\Phi := \Psi_1^{-1}\circ\Psi_2$, we see that $\Psi_1^{2,0} = \Psi_2^{2,0}$ if and only if $\Phi^{2,0} = \operatorname{id}$, and similarly for the Chow action.
\end{remark}

This reduction motivates the following terminology.

\begin{definition}
Let $\Psi$ be a derived (anti-)autoequivalence of $\Db(X,\alpha)$. We call $\Psi$ \emph{symplectic} if $\Psi^{2,0} = \id$, and \emph{anti-symplectic} if $\Psi^{2,0} = -\id$.
\end{definition}

Consequently, \Cref{conj:autoeq} is equivalent to proving that every symplectic autoequivalence acts trivially on $\CH_0(X)_{\alb}$. For anti-symplectic autoequivalences, one can again appeal to \Cref{conj:autoeq}: by comparing such an autoequivalence $\Psi$ with the shift functor
\[
[1]:\Db(X,\alpha)\to \Db(X,\alpha),
\]
one expects $\Psi^{\CH_0}=-\operatorname{id}$. 

This derived version significantly strengthens classical Bloch-type results. Traditionally, Bloch's conjecture is studied via correspondences arising from cycles or geometric automorphisms (e.g., \cite{BlochKasLieberman, Voisin2014, PW16} for surfaces with $p_g=0$, or \cite{VoisinSympK3, HuybrechtsSymplectic, DuLiu2021} for special geometric situations on K3 surfaces). The derived approach goes much further: Fourier--Mukai kernels need not come from automorphisms, and they live naturally on twisted products. Moreover, \emph{anti-equivalences} have no direct counterpart in the classical geometric setting.


Previously, this derived formulation had been established in the following cases:
\begin{itemize}
    \item For untwisted K3 surfaces, Huybrechts \cite[Theorem 3]{Huybrechts10} proved that if two derived equivalences induce the same cohomological action, then their actions on the Albanese kernel coincide. In particular, the conjecture holds for spherical twists, whose cohomological action is a $(-2)$-reflection.
    \item For untwisted K3 surfaces with Picard number greater than $2$, \Cref{conj:autoeq} with its anti-symplectic analogue was further confirmed in \cite[Theorem 1.3]{LiYuZhang}.
\end{itemize}
In this paper, we construct the analogous action $\Psi^{\mathrm{CH}_0}$ for \emph{twisted abelian surfaces} and verify the conjecture in full generality. 
Our first main result is the following.

\begin{maintheorem} \label{thm:main}
Let $(X,\alpha)$ be a twisted abelian surface. For any derived (anti‑)autoequivalence $\Psi$ of $\Db(X,\alpha)$, we have 
\[
\Psi^{\mathrm{CH}_0} = 
\begin{cases}
\phantom{-}\operatorname{id} & \text{if $\Psi$ is a symplectic equivalence, or an anti‑symplectic anti‑equivalence};\\[2pt]
-\operatorname{id} & \text{if $\Psi$ is an anti‑symplectic equivalence, or a symplectic anti‑equivalence}.
\end{cases}
\]
In particular, \Cref{conj:autoeq} holds for (anti-)equivalences between twisted abelian surfaces. 
\end{maintheorem}

\Cref{thm:main} actually goes beyond \Cref{conj:autoeq}. 
A genuinely new point is the treatment of anti‑equivalences. For a twisted surface $(X,\alpha)$, the derived dual functor provides a natural anti‑equivalence
$\DD \colon \Db(X,\alpha)^{\mathrm{op}} \to \Db(X,-\alpha)$.
Given an (anti‑)symplectic anti‑autoequivalence $\Psi$ of $\Db(X,\alpha)$, the composition
\[
\DD \circ \Psi \colon \Db(X,\alpha)\to \Db(X,\alpha)^{\rm op}\to  \Db(X,-\alpha),
\]
is an equivalence, and the action of $\Psi$ on $\CH_0(X)_{\alb}$ coincides with that of $\DD \circ \Psi$.
If $\alpha = -\alpha$ in $\Br(X)$ (i.e., $\alpha$ has order $1$ or $2$), then $\Db(X,-\alpha) \simeq \Db(X,\alpha)$, and the $\CH_0$‑action of $\DD \circ \Psi$ (which must be $\pm \id$) is governed by \Cref{conj:autoeq}.
However, when the order of $\alpha$ is strictly larger than $2$, the categories $\Db(X,\alpha)$ and $\Db(X,-\alpha)$ can be different, and no direct comparison argument via \Cref{conj:autoeq} is available for anti‑equivalences of \emph{either} symplectic or anti‑symplectic kind.

\Cref{thm:main} is the input for the hyperkähler applications. The mechanism of the paper is to control zero-cycles on twisted modular varieties of $\mathrm{Kum}_n$-type through derived equivalences of the underlying twisted abelian surfaces. In the symplectic case, this leads to Bloch's conjecture for birational automorphisms. The anti-symplectic case requires a finer formulation, discussed in \Cref{subsection:anti-symplectic}, involving the expected Beauville--Voisin filtration and alternating signs on its graded pieces.

\subsection{Bloch's conjecture for symplectic birational automorphisms}

We now explain the main hyperk\"ahler application of \Cref{thm:main}. For a hyperk\"ahler variety, the analogue of a symplectic autoequivalence is a birational automorphism acting trivially on the holomorphic symplectic form. The corresponding Bloch principle predicts that such a symmetry should act trivially on zero-cycles.

\begin{conjecture}[Bloch's conjecture for symplectic birational automorphisms] \label{conj:main}
Let $Y$ be a smooth projective hyperk\"ahler variety. Let $f \in \Bir(Y)$ be a symplectic birational automorphism. Then the induced map
\[
f^{\CH_0} \colon \CH_0(Y) \longrightarrow \CH_0(Y)
\]
is the identity.
\end{conjecture}
For hyperk\"ahler varieties of $\mathrm{K3}^{[n]}$-type arising as Bridgeland moduli spaces, \Cref{conj:main} was proved in \cite{LiYuZhang} under the assumptions that the Picard number is at least $4$ and that the underlying K3 surface is untwisted. These restrictions were due to the K3 version of \Cref{conj:autoeq} not being fully established in general. The case of $\mathrm{Kum}_n$-type is the next natural deformation type, where the underlying surface is abelian. The required surface input is precisely what \Cref{thm:main} provides---now unconditionally of Picard number and in the twisted setting.

Recall that a hyperk\"ahler variety $Y$ is of $\mathrm{Kum}_n$-type ($n\ge 2$) if it is deformation equivalent to the classical generalized Kummer variety $$K_n(X)=\ker(X^{[n+1]}\to X)$$ constructed by Beauville \cite{BeauvilleHK}, where $X$ is an abelian surface. Following \cite{MM15}, we call $Y$ \emph{twisted modular} if $Y$ is birational to the Albanese fiber of a Bridgeland moduli space of some twisted abelian surface.

Our main hyperk\"ahler consequence is the following. It provides the first evidence for the Bloch conjecture in the $\mathrm{Kum}_n$-type setting and applies, in particular, to the geometrically important Lagrangian-fibered case.

\begin{maintheorem} \label{thm:main2}
Let $Y$ be a hyperk\"ahler variety of $\mathrm{Kum}_n$-type and let $\phi$ be a symplectic birational automorphism.  Then \Cref{conj:main} holds if $Y$ is twisted modular. In particular,  this holds whenever $Y$ admits a birational Lagrangian fibration.
\end{maintheorem}

The final assertion is a lattice-theoretic consequence of the Markman--Mukai description: admitting a birational Lagrangian fibration implies the existence of an isotropic vector in the algebraic Markman--Mukai lattice; see \Cref{def:markman-mukai}. By \cite[Proposition~3.1]{CKKM}, this condition yields that $Y$ is twisted modular.

\begin{remark}\label{rem:finite-order-bloch}
In \cite[Theorem 5]{Vial}, Vial established \Cref{conj:main} for finite-order symplectic \emph{automorphisms} on classical generalized Kummer varieties \(K_n(X)\). 
Using the classification of finite-order symplectic birational self-maps on \(\Kum_n\)-type manifolds obtained in \cite[\S 3]{Dutta+26}, one can extend this result to finite-order symplectic \emph{birational} automorphisms in many cases. Actually, \Cref{conj:main} holds unless \(f\) falls into the exceptional situations described in  \cite[Corollary 3.12]{Dutta+26}. 
\end{remark}

The proof of \Cref{thm:main2} has two separate steps. First, one must lift a birational automorphism of the Albanese fiber to a derived (anti-)autoequivalence of the underlying twisted abelian surface $(X,\alpha)$. Second, after such a lift has been constructed, \Cref{thm:main} controls its action on $\CH_0(X)_{\alb}$, and a Marian--Zhao type comparison transfers this control to $\CH_0(Y)$. Thus the real bridge from the surface theorem to the hyperk\"ahler theorem is the lifting step.

The missing input for this lifting step is a Torelli theorem in the even cohomology of a twisted abelian surface. For twisted abelian varieties of arbitrary dimension, Orlov~\cite{Orlov} and Li--Lu--Tang~\cite{LLT25} prove derived Torelli theorems using the Orlov lattice
\[\widehat{\rH}(X)= \rH^1(X,\ZZ) \oplus \rH^{2g-1}(X,\ZZ)\] 
or its twist; see \Cref{thm:de-orlov}. 
For abelian surfaces, however, an additional refined structure is available: the twisted Mukai lattice (the even cohomology) provides a more convenient framework for controlling derived equivalences. Moreover, for the $\mathrm{Kum}_n$-type application, the monodromy and Markman--Mukai lattice naturally live in the twisted Mukai lattice. A Torelli theorem in this lattice is therefore needed.

This even-cohomology formulation has an additional subtlety. Unlike the K3 surface case, not every Hodge isometry between twisted Mukai lattices of abelian surfaces is induced by a derived (anti-)equivalence. We isolate the exact extra condition by introducing \emph{admissible} Hodge isometries. With this condition, we recover a perfect analogue of the K3 surface Torelli theorem.

\begin{maintheorem}[\Cref{thm:derived-torelli}] \label{thm:torelli-C}
    Let $(X,\alpha)$ and $(X',\alpha')$ be twisted abelian surfaces.  There exists an admissible Hodge isometry between their twisted Mukai lattices
\[
h \colon \widetilde{\rH}(X,\alpha,\mathbb{Z}) \xrightarrow{\sim} \widetilde{\rH}(X',\alpha',\mathbb{Z})
\]
if and only if there exists a derived (anti-)equivalence
\[
\Psi \colon \Db(X,\alpha) \xrightarrow{\sim} \Db(X',\alpha') \quad \text{or} \quad \Psi \colon \Db(X,\alpha) \xrightarrow{\sim} \Db(X',\alpha')^{\rm op}
\]
such that $\Psi$ induces $h$.
\end{maintheorem}

\subsection{Beauville--Voisin filtrations and anti-symplectic actions} \label{subsection:anti-symplectic}

More generally, the expected behaviour of anti-symplectic birational automorphisms should be formulated at the level of a filtration on $\CH_0(Y)$, in line with the conjectural Beauville--Voisin filtration. Voisin \cite{VoisinHK} proposed an explicit candidate for such a filtration on the Chow group of zero-cycles.
Let $Y$ be a smooth projective hyperk\"ahler variety of dimension $2n$, and define
\[
\bS_i\CH_0(Y) = \langle y\in Y\mid \dim O_y\geq n-i \rangle,
\]
where $O_y\subset Y$ denotes the set of points rationally equivalent to $y$. 
This defines an increasing filtration $\bS_\bullet \CH_0(Y)$.

\begin{conjecture}[Bloch's conjecture for anti-symplectic birational automorphisms] \label{conj:anti-symplectic}
Let $f\in \Bir(Y)$ be an anti-symplectic birational automorphism.  Then $f_*$ preserves Voisin's filtration $\bS_\bullet\CH_0(Y)$ and acts on the graded
pieces by
\[
f_*|_{\operatorname{gr}^{\bS}_i\CH_0(Y)}
=
(-1)^i\id .
\]
\end{conjecture}

Verifying \Cref{conj:anti-symplectic} would provide nontrivial evidence for the Beauville--Voisin conjecture itself.
For twisted modular varieties of $\mathrm{K3}^{[n]}$-type, the anti-symplectic case follows easily from the symplectic case \cite{LiYuZhang}. 
For $\mathrm{Kum}_n$-type varieties, however, the situation is much more subtle. The main difficulty lies in the lack of an analogue of the ``SYZ = Voisin'' result. More precisely, there exists an O'Grady filtration on the Chow group of a (twisted) K3 or abelian surface $X$, which induces a Shen--Yin--Zhao (SYZ) type filtration on the Albanese fibers $Y$ of Bridgeland moduli spaces over $X$. This leads to the natural question of whether the two filtrations coincide:
\[
\mathbf{S}_\bullet \CH_0(Y)=\mathbf{S}_\bullet^{\rm SYZ} \CH_0(Y)
\]
For $\mathrm{K3}^{[n]}$-type varieties, this equality is valid, and the proof relies on constructing constant cycle varieties in a series of works:
\begin{itemize}
    \item In \cite[Corollary 1.11]{Voisin15}, Voisin proved that $\mathbf{S}_\bullet \CH_0(S^{[n]})=\mathbf{S}_\bullet^{\rm SYZ} \CH_0(S^{[n]})$, where $S^{[n]}$ is the Hilbert scheme of points on a K3 surface $S$.
    \item In \cite[Theorem 1.1]{LZ23} and \cite[Theorem 5.3]{CLZZ:2024}, the authors showed that the two filtrations coincide when $Y$ is a Bridgeland moduli space over a (twisted) K3 surface. The method is based on Voisin's result and degenerate loci of (twisted) vector bundles. 
\end{itemize}
When $X$ is a twisted abelian surface, the equality remains open even at the first nontrivial step. As evidence, we prove it for the classical generalized Kummer sixfolds.
\begin{theorem}[\Cref{thm:SYZ-sixfold}]
If $Y = K_3(X)$, then $\mathbf{S}_\bullet \CH_0(Y) = \mathbf{S}_\bullet^{\rm SYZ} \CH_0(Y)$.
\end{theorem}

We prove that ``SYZ = Voisin'' leads to \Cref{conj:anti-symplectic} for twisted modular $\mathrm{Kum}_n$-type varieties. Moreover, by using Floccari's construction \cite{Fl24}, we give a proof which does not rely on ``SYZ = Voisin'' for twisted modular sixfolds.

\begin{maintheorem}[\Cref{prop:SYZ-give-anti} and \Cref{thm:modular-sixfold}] \label{thm:antisymplectic-sixfold}
Let $Y$ be a hyperk\"ahler variety of~$\mathrm{Kum}_n$-type, and let $f \in \Bir(Y)$ be an anti-symplectic birational automorphism. Suppose that one of the following conditions holds:
\begin{enumerate}
    \item ``SYZ = Voisin'' holds (see \Cref{conj:SYZ}) and $Y$ is twisted modular;
    \item $n=3$ and $Y$ is twisted modular.
\end{enumerate}
Then \Cref{conj:anti-symplectic} holds.
\end{maintheorem}

\begin{flexible}{Organization of the paper}
    \begin{itemize}
        \item \textbf{Section 2} recalls twisted derived categories, \textbf{B}-fields, twisted Mukai and Orlov lattices, Fourier--Mukai actions, and the equivariant model for twisted derived categories.
        \item \textbf{Section 3} proves derived Torelli theorems for twisted abelian surfaces. After reviewing the Kapustin--Orlov conjecture, we introduce admissible cohomological actions and prove a derived Torelli theorem (\Cref{thm:torelli-C}) using the twisted Mukai lattice.
        \item \textbf{Section 4} develops the descent of autoequivalences to twisted Kummer surfaces. We recall the geometry of twisted Kummer surfaces, establish a twisted BKR equivalence, and show that any derived equivalence can be adjusted (by an element of $A_{X,\alpha}$) so that it descends to the Kummer surface.
        \item \textbf{Section 5} proves \Cref{thm:main}. The descent from Section 4 identifies the Albanese kernel of a twisted abelian surface with that of the associated twisted Kummer surface, reducing the claim to the known K3 case.
        \item \textbf{Section 6} studies Bridgeland moduli spaces on twisted abelian surfaces and their Albanese fibers. We recall the monodromy and wall-crossing results needed to lift birational automorphisms of $K_{\sigma}(X,\alpha,\mathbf{v})$ to derived (anti-)equivalences. Together with a criterion of Marian--Zhao, we prove \Cref{thm:main2}.
        \item \textbf{Section 7} discusses Beauville--Voisin filtrations on $\CH_0(K_{\sigma}(X,\alpha,\mathbf{v}))$. We introduce a Shen--Yin--Zhao type filtration, relate it to O'Grady's filtration on the underlying abelian surface, and prove equality with Voisin's filtration for sixfolds, leading to the anti-symplectic application \Cref{thm:antisymplectic-sixfold}.
    \end{itemize}
\end{flexible}

\begin{flexible}{Notations and conventions}
Throughout this paper, we work over the complex numbers and adopt the following conventions.
\begin{itemize}
    \item All Chow groups in this paper are taken with rational coefficients.
    \item For a point $z\in X$, we use $\chow{z}\in\CH_0(X)$ to denote the rational equivalence class of $z$.
\end{itemize}
\end{flexible}

\begin{flexible}{Acknowledgements}
The authors would like to thank Ziwei Lu and Hanfei Guo for helpful discussions and valuable comments on an earlier draft of this work. Z.~Li is supported by the NSFC grants (No.~12171090 and No.~12425105) and the Shanghai Pilot Program for Basic Research (No.~21TQ00). Z.~Li is also a member of LMNS.
Z.~Chen and R.~Zhang are supported by the NSFC grant (No.~12121001). R.~Zhang is also supported by the China Postdoctoral Science Foundation grant (No.~2025M783085).
\end{flexible}

\section{Preliminaries on twisted derived categories}
 This section collects the lattice-theoretic and equivariant descriptions of twisted derived categories that will be used in the Torelli arguments of \Cref{sec:derived-torelli} and in the descent argument of \Cref{sec:kummer-descent}.

\subsection{Twisted derived categories and twisted Mukai lattices}

Let \(X\) be a smooth projective variety over \(\CC\).  In this paper, a \emph{twisted variety} is a pair \((X,\alpha)\) where \(\alpha\in\Br(X)=\rH^2_{\et}(X,\GG_m)\) is a Brauer class.

\begin{definition}
Let \((X,\alpha)\) be a twisted variety.  Denote by \(\Coh(X,\alpha)\) the abelian category of \(\alpha\)-twisted coherent sheaves in the sense of \cite{Caldararu}.  
The bounded derived category of \(\Coh(X,\alpha)\) is denoted by \(\Db(X,\alpha)\).  
\end{definition}

For two twisted varieties \((X,\alpha)\) and \((X',\alpha')\),   a derived equivalence
\[
\Psi:\Db(X,\alpha)\xrightarrow{\sim}\Db(X',\alpha')
\]
can be represented by a Fourier--Mukai kernel \(\cP\) which is a perfect complex in \(\Db(X\times X', \alpha^{-1}\boxtimes \alpha')\) (see \cite[Theorem 1.1]{CS07}). In the rest of this paper we specialize to the case where the underlying variety \(X\) is either a K3 surface or an abelian surface.  The derived categories \(\Db(X,\alpha)\) then enjoy additional properties, which will be exploited in the subsequent sections.

Assume \(X\) is either a K3 surface or an abelian surface, and \(\alpha\in\Br(X)\).  For the untwisted case (\(\alpha=0\)), the \emph{Mukai lattice} of \(X\) is defined as
\[
\widetilde{\rH}(X,\mathbb{Z}) \coloneqq \rH^0(X,\mathbb{Z}) \oplus \rH^2(X,\mathbb{Z}) \oplus \rH^4(X,\mathbb{Z}),
\]
endowed with the Mukai pairing
\begin{equation}\label{eq:mukaipairing}
\big\langle (r_1, b_1, s_1),\, (r_2, b_2, s_2) \big\rangle \coloneqq b_1\cdot b_2 - r_1 s_2 - r_2 s_1,
\end{equation}
and a pure \(\mathbb{Z}\)-Hodge structure of weight \(2\).

In the twisted case, we need the notion of a \textbf{B}-field lift of \(\alpha\).

\begin{definition}
Choose an element \(B\in \rH^2(X,\mathbb{Q})\) whose image in $\rH^2(X,\QQ/\ZZ)$ maps to $\alpha$ via the short exact sequence
\[
0 \longrightarrow \NS(X) \otimes \QQ/\ZZ \longrightarrow \rH^2(X,\QQ/\ZZ) \longrightarrow \Br(X) \longrightarrow 0.
\]
Such a lift \(B\) is called a \emph{\textbf{B}-field lift} of \(\alpha\).  Using \(B\), we define the \emph{twisted Mukai lattice} of \((X,\alpha)\) as
\[
\widetilde{\rH}(X,\alpha,\mathbb{Z}) \coloneqq \exp(B)\cdot \widetilde{\rH}(X,\mathbb{Z}) \;\subset\; \widetilde{\rH}(X,\mathbb{Z})\otimes_{\mathbb{Z}}\mathbb{Q},
\]
which is a full rank lattice isomorphic to \(\widetilde{\rH}(X,\mathbb{Z})\).  
\end{definition}

As explained in \cite{Caldararu, LZ25}, the twisted Mukai lattice is independent of the choice of the \textbf{B}-field up to ``nice'' isometries (see \Cref{example:B-fields} for details); we therefore denote it by \(\widetilde{\rH}(X,\alpha,\mathbb{Z})\).  The algebraic part is defined as \(\widetilde{\rH}_{\rm alg}(X,\alpha,\ZZ) \coloneqq \widetilde{\rH}^{1,1}(X,\alpha,\mathbb{Z})\).

Let \(\rK_0(X,\alpha)\) be the Grothendieck group of \(\Db(X,\alpha)\).  A fundamental tool is the \emph{twisted Chern character}
\[
\ch_{X,\alpha}: \rK_0(X,\alpha) \longrightarrow \widetilde{\rH}(X,\alpha,\mathbb{Q}),
\]
constructed in \cite[Proposition 1.2]{HS05}.  This map is compatible with the Mukai pairing via the Euler pairing on \(\rK_0(X,\alpha)\).  In the untwisted case \(\alpha=0\), we recover the usual Chern character \(\ch: \rK_0(X)\to \widetilde{\rH}(X,\mathbb{Q})\).
By further composing with the cycle class map, one can obtain an integral Mukai map (\cf \cite{YoshiokaModuli,HS05})
\[
v(E) \coloneqq \operatorname{cl} \bigl(\operatorname{ch}_{X,\alpha}(E)\bigr) \in \widetilde{\rH}_{\mathrm{alg}}(X,\alpha,\mathbb{Z}).
\]

\subsection{Brauer classes and the twisted Orlov lattice for abelian surfaces}

When \(X\) is an abelian surface, the group structure provides additional tools to study the Brauer class, which naturally leads to a \textbf{B}-field on the odd cohomology \(\rH^1(X,\mathbb{Q})\oplus \rH^3(X,\mathbb{Q})\).  This also plays a crucial role in the study of twisted derived equivalence.

More precisely, for a Brauer class \(\alpha\in\Br(X)[r]\) of exact order \(r\), the Kummer exact sequence for the multiplication-by-\(r\) map \(\mathbf{r}_X:X\to X\) induces
\[
0 \longrightarrow \NS(X)/r\NS(X) \xlongrightarrow{\iota} \Hom(\wedge^2 X[r], \mu_r) \xlongrightarrow{\delta} \Br(X)[r] \longrightarrow 0,
\]
where \(X[r]\) denotes the group of \(r\)-torsion points.  Hence \(\alpha\) can be represented by an alternating bilinear pairing
\[
e_\alpha \colon X[r]\times X[r] \longrightarrow \mu_r,
\]
unique modulo the commutator forms \(e^{L^{\otimes r}}|_{X[r]}\) for \(L\in\Pic(X)\).

Using the Weil pairing, \(e_\alpha\) corresponds to a skew-symmetric homomorphism
\begin{equation}
\phi_\alpha \colon X[r] \longrightarrow \widehat{X}[r],\quad 
e_\alpha(\sigma_1,\sigma_2)=\langle\sigma_1,\phi_\alpha(\sigma_2)\rangle,
\end{equation}
where \(\langle-,-\rangle \colon X[r]\times\widehat{X}[r]\to\mu_r\) is the canonical Weil pairing.  
Because the map \(\rH^2(X[r],\mathbb{G}_m)\to\Hom(\wedge^2 X[r],\mu_r)\) is surjective, there exists a normalized \(2\)-cocycle \(a\in\mathsf{Z}^2(X[r],\mathbb{G}_m)\) such that
\[
e_\alpha(\sigma_1,\sigma_2)=\frac{a_{\sigma_1,\sigma_2}}{a_{\sigma_2,\sigma_1}}.
\]
We say that \(\alpha\) is represented by the cocycle \(a\).  This cocycle description will be used when we discuss equivariant categories.

Let us now review the concept of twisted Orlov lattice. The classical Orlov lattice is the unimodular lattice
\[
\widehat{\mathrm{H}}(X,\mathbb{Z}) := \mathrm{H}^1(X,\mathbb{Z}) \oplus \mathrm{H}^1(\widehat{X},\mathbb{Z}),
\]
equipped with the symmetric bilinear form
\[
q\bigl((\xi,\eta), (\xi',\eta')\bigr) := \eta(\xi') + \eta'(\xi),
\]
where \(\eta(\xi')\) denotes the natural pairing between \(\mathrm{H}^1(X,\mathbb{Q})^*\) and \(\mathrm{H}^1(X,\mathbb{Q})\). Let \(J_X\) denote the complex structure on \(\mathrm{H}^1(X,\mathbb{R})\). Then the dual complex structure on \(\mathrm{H}^1(\widehat{X},\mathbb{R}) \cong \mathrm{H}^1(X,\mathbb{R})^*\) is given by \(-J_X^{\top}\). The product complex structure on \(\widehat{\mathrm{H}}(X,\mathbb{Z}) \otimes \mathbb{R}\) is therefore
\[
\mathcal{J}_0 := \begin{pmatrix}
J_X & 0 \\
0 & -J_X^{\top}
\end{pmatrix}.
\]

For a twisted abelian variety \((X,\alpha)\), choose a \textbf{B}-field lift \(B \in \mathrm{H}^2(X,\mathbb{Q})\) of \(\alpha\). Via the isomorphism \(\mathrm{H}^2(X,\mathbb{Q}) \cong \operatorname{Hom}_{\mathbb{Q}}(\mathrm{H}^1(X,\mathbb{Q}), \mathrm{H}^1(\widehat{X},\mathbb{Q}))\), we view \(B\) as a linear map
\[
B : \mathrm{H}^1(X,\mathbb{Q}) \longrightarrow \mathrm{H}^1(\widehat{X},\mathbb{Q}), \qquad B^{\top} = -B.
\]
We can define an action on \(\mathrm{H}^1(X,\mathbb{Q}) \oplus \mathrm{H}^1(\widehat{X},\mathbb{Q})\) by
\[
\exp(B)(\xi, \eta) := \bigl(\xi,\; \eta + B(\xi)\bigr).
\]
This upper-triangular automorphism preserves the bilinear form \(q\) because \(B\) is skew-symmetric.

\begin{definition}[Twisted Orlov lattice]\label{def:twist-Orlov}
The \emph{twisted Orlov lattice} of \((X,\alpha)\) is
\[
\widehat{\mathrm{H}}(X,\alpha,\mathbb{Z}) := \exp(B)\bigl(\mathrm{H}^1(X,\mathbb{Z}) \oplus \mathrm{H}^1(\widehat{X},\mathbb{Z})\bigr) \;\subset\; \mathrm{H}^1(X,\mathbb{Q}) \oplus \mathrm{H}^1(\widehat{X},\mathbb{Q}).
\]
It is a full-rank sublattice, isomorphic to \(\widehat{\mathrm{H}}(X,\mathbb{Z})\) via \(\exp(B)\). The complex structure on \(\widehat{\mathrm{H}}(X,\alpha,\mathbb{Z}) \otimes \mathbb{R}\) is given by conjugation:
\[
\mathcal{J}_\alpha := \exp(B) \circ \mathcal{J}_0 \circ \exp(B)^{-1}
= \begin{pmatrix}
J_X & 0 \\
B J_X + J_X^{\top} B & -J_X^{\top}
\end{pmatrix},
\]
and the symmetric bilinear form \(q\) is inherited from the natural pairing.
\end{definition}
The above construction is independent of the choice of the \textbf{B}-field lift up to isomorphism: a different lift \(B' = B + \ell\) with \(\ell \in \mathrm{H}^2(X,\mathbb{Z})\) corresponds to tensoring by a (topological) line bundle, which induces an isomorphism of twisted Orlov lattices preserving the bilinear form.

\subsection{Fourier--Mukai actions on Chow groups and cohomology} \label{subsection:FM-action}

Let $(X,\alpha)$ and $(X',\alpha')$ be twisted abelian surfaces, whose Todd classes are trivial. Let \(\Psi_{\cE}: \Db(X,\alpha) \to \Db(X',\alpha')\) be a derived equivalence with Fourier–Mukai kernel \(\cE \in \Db(X\times X', \alpha^{-1}\boxtimes \alpha')\).

\begin{definition}
The induced action of $\Psi_{\cE}$ on Chow groups is given by the correspondence
\begin{equation} \label{eq:fm-action-chow}
\Psi_{\cE}^{\CH_*}: \CH_*(X) \longrightarrow \CH_*(X'),\qquad 
z \longmapsto \pi_{2*}\bigl( \pi_1^*(z) \cdot \ch_{X\times X'}(\cE) \bigr),
\end{equation}
where \(\pi_1,\pi_2\) are the projections and \(\ch_{X\times X'}(\cE)\) is the twisted Chern character of the kernel. 
\end{definition}
Since \(\CH_*(X)\) is spanned by the twisted Chern characters \(\ch_{X,\alpha}(F)\) for \(F\in\Db(X,\alpha)\), this map can be simply realized as 
\[
\Psi_{\cE}^{\CH_*}: \CH_*(X) \longrightarrow \CH_*(X'),\qquad 
\ch_{X,\alpha}(F) \longmapsto \ch_{X',\alpha'}(\Psi_{\cE}(F))
\]
by Grothendieck-Riemann-Roch.

We now pass to cohomology. For an abelian surface $X$, the total cohomology $\rH^*(X,\mathbb{Q})$ splits naturally into even and odd parts:
\[
\rH^{\rm even}(X,\mathbb{Q}) = \rH^0(X,\mathbb{Q})\oplus \rH^2(X,\mathbb{Q})\oplus \rH^4(X,\mathbb{Q}),\quad
\rH^{\rm odd}(X,\mathbb{Q}) = \rH^1(X,\mathbb{Q})\oplus \rH^3(X,\mathbb{Q}).
\]
Via cup product we have isomorphisms $\bigwedge^{i}\rH^1(X,\mathbb{Q})\cong\rH^{i}(X,\mathbb{Q})$ and consequently,
\[
\rH^{\rm even}(X,\mathbb{Q}) \cong \bigwedge^{\rm even}\rH^1(X,\mathbb{Q}),\quad
\rH^{\rm odd}(X,\mathbb{Q}) \cong \bigwedge^{\rm odd}\rH^1(X,\mathbb{Q}).
\]

\begin{definition}
The cohomological action of $\Psi_{\cE}$ is the map
\begin{equation} \label{eq:fm-action-cohomology}
\Psi_{\cE}^{\rH}\colon \rH^*(X,\mathbb{Q}) \longrightarrow \rH^*(X',\mathbb{Q}),\qquad 
z \longmapsto \pi_{2*}\bigl( \pi_1^*(z) \cdot \ch_{X\times X'}(\cE) \bigr).
\end{equation}
Since $\ch_{X\times X'}(\cE)$ is an even cohomology class, this map preserves the parity of cohomology:
\[
\Psi_{\cE}^{\rH} = \Psi_{\cE}^{\rH^{\mathrm{even}}} \oplus \Psi_{\cE}^{\rH^{\mathrm{odd}}}.
\]
\end{definition}

For a twisted abelian surface $(X,\alpha)$ with a $\mathbf{B}$-field lift $B$, the twisted Mukai lattice (with rational coefficients) is identified with the even cohomology:
\[
\widetilde{\rH}(X,\alpha,\mathbb{Z})\otimes \mathbb{Q} \;\cong \; \rH^{\rm even}(X,\mathbb{Q}),
\]
and similarly the twisted Orlov lattice is identified with the odd cohomology:
\[
\widehat{\rH}(X,\alpha,\mathbb{Z})\otimes \mathbb{Q} \;\cong\; \rH^{\rm odd}(X,\mathbb{Q}).
\]
Restricting \cref{eq:fm-action-cohomology} to the even part yields an integral Hodge isometry on the twisted Mukai lattices:
\begin{equation}
\Psi_{\cE}^{\widetilde{\rH}}: \widetilde{\rH}(X,\alpha,\mathbb{Z}) \longrightarrow \widetilde{\rH}(X',\alpha',\mathbb{Z}),\qquad 
v(F) \longmapsto v\bigl(\Psi_{\cE}(F)\bigr).
\end{equation}
Similarly, the induced map on odd cohomology restricts to an integral Hodge isometry
\[
\Psi_{\cE}^{\widehat{\rH}}: \widehat{\rH}(X,\alpha,\mathbb{Z}) \longrightarrow \widehat{\rH}(X',\alpha',\mathbb{Z}),\qquad 
w \longmapsto \pi_{2*}\bigl( \pi_1^*(w) \cdot \ch_{X\times X'}(\cE) \bigr),
\]
where $w$ is viewed as an odd cohomology class via the identification $\widehat{\rH}(X,\alpha,\mathbb{Q})\cong \rH^{\mathrm{odd}}(X,\mathbb{Q})$.
Such an isometry can be realized as
\begin{equation}
\Psi_{\cE}^{\widehat{\rH}}: \widehat{\rH}(X,\alpha,\mathbb{Z}) \longrightarrow \widehat{\rH}(X',\alpha',\mathbb{Z}),\qquad  w \longmapsto (f_{\cE}^*)^{-1}(w).
\end{equation}
Here $f_{\cE}$ is the symplectic isomorphism associated with $\Psi_{\cE}$; see \Cref{thm:de-orlov}.

Choose \textbf{B}-field lifts $B$ and $B'$ of $\alpha$ and $\alpha'$. Then $\Psi^{\widetilde{\rH}}$ sends the twisted $(2,0)$-line $\CC\cdot \exp(B)\sigma_X$ to $\CC\cdot \exp(B')\sigma_{X'}$. We define $\Psi^{2,0}$ by the scalar relation
\begin{equation}
\Psi^{\widetilde{\rH}}(\exp(B)\sigma_X)=\exp(B')\Psi^{2,0}(\sigma_X).
\end{equation}
This action determines the induced Hodge isometry on the twisted transcendental lattices
\[
\rT(X,\alpha) \coloneqq \widetilde{\rH}_{\rm alg}(X,\alpha,\ZZ)^\perp\subset \widetilde{\rH}(X,\alpha,\ZZ),
\]
because $\rT(X,\alpha)_{\QQ}$ is the smallest rational Hodge substructure containing the twisted $(2,0)$-line. Thus the formulation of \Cref{conj:autoeq} is equivalent to
\[
\Psi_1^{\tr} = \Psi_2^{\tr} \quad \Longleftrightarrow \quad \Psi_1^{\CH_0} = \Psi_2^{\CH_0},
\]
where $\Psi_i^{\tr} \colon \rT(X,\alpha) \to \rT(X',\alpha')$ are the induced Hodge isometries on the transcendental lattices as above.

These cohomological actions are compatible with the Chow action via the cycle class map, and will be frequently used later.

\subsection{Equivariant models for twisted derived categories} \label{subsection:equivariant-category}

Following Beckmann--Oberdieck \cite[\S 2.1]{BO23}, an action of a finite group \(G\) on a triangulated category \(\cD\) consists of two collections of data \((\rho,a)\):

\begin{enumerate}[label=(\roman*)]
    \item For each \(g \in G\), an \emph{autoequivalence} \(\rho_g \colon \cD \to \cD\),
    \item For each pair \(g, h \in G\), a \emph{natural isomorphism} \(a_{g,h} \colon \rho_g \circ \rho_h \to \rho_{gh}\) satisfying the cocycle condition.
\end{enumerate}

In our specific setting, let \((X,\alpha)\) be a twisted abelian surface with Brauer class \(\alpha\) of order \(r\).  Choose a normalized \(2\)-cocycle \(a\in\mathsf{Z}^2(G,\GG_m)\) representing \(\alpha\), where \(G = X[r]\) is the group of \(r\)-torsion points.  For each \(g\in G\), denote by \(t_g: X\to X\) the translation map \(x\mapsto x+g\). Define
\[
\rho_g = \bL t_g^* : \Db(X) \to \Db(X),\qquad g\in G,
\]
the derived pullback along translation by \(g\).  Then the \emph{equivariant derived category} \(\Db(X)_{G,\rho,a}\) consists of objects \((E,\theta)\) where \(E\in\Db(X)\) and \(\theta=\{\theta_g:\rho_g(E)\to E\}_{g\in G}\) satisfies a compatibility condition involving the cocycle \(a\) (see \cite[\S 2]{LLT25} for details).  Morphisms are those that commute with the linearizations.

The following result is the key input for the equivariant model used in the Kummer descent.

\begin{proposition}[{\cite[Corollary 2.6]{LLT25}}]
For a twisted abelian surface \((X,\alpha)\) with \(\alpha\) represented by the cocycle \(a\in\mathsf{Z}^2(G,\GG_m)\), there is a canonical equivalence
\[
\Db(X,\alpha) \simeq \Db(X)_{G,\rho,a}.
\]
\end{proposition}

In this paper, the derived category of a twisted abelian surface will be analyzed through this equivariant model.

\section{Derived Torelli theorems for twisted abelian surfaces} \label{sec:derived-torelli}

Torelli-type questions have been studied for (twisted) abelian varieties in various settings; see for example \cite{Orlov, YoshiokaBir, LZ25}. For abelian surfaces, two types of derived Torelli theorems are known, depending on which cohomological lattice one uses:
\begin{itemize}
    \item Using the \emph{Orlov lattice} (odd cohomology \(\rH^1(X,\mathbb{Z}) \oplus \rH^1(\widehat{X},\mathbb{Z})\)), see \cite{Orlov}.
    \item Using the \emph{Mukai lattice} (even cohomology \(\rH^{\mathrm{even}}(X,\mathbb{Z})\)), partial results are available; see e.g. \cite{LZ25}.
\end{itemize}
In this section, we review the Orlov-lattice description of twisted derived equivalences and then prove the Mukai-lattice Torelli statement that will be used to lift birational actions in later sections.

\subsection{The Kapustin--Orlov conjecture via twisted symplectic abelian varieties} \label{subsection:symplectic}

Let $(X,\alpha)$ be a twisted abelian surface, where $\alpha\in\Br(X)$ is a Brauer class of order $r$.  Using the Weil pairing, one associates to $(X,\alpha)$ a canonical symplectic abelian variety:
\begin{equation}\label{eq:symplectic-variety}
A_{X,\alpha} \coloneqq (X \times \widehat{X}) / \Gamma_\phi,
\end{equation}
where $\phi : X[r] \to \widehat{X}[r]$ is the skew-symmetric homomorphism induced by the Brauer class $\alpha$, and $\Gamma_\phi\subseteq X[r]\times \widehat{X}[r]$ is the graph of $\phi$.  When $\alpha=0$ (untwisted case), we have $A_{X}=X\times\widehat{X}$.

The symplectic abelian variety $A_{X,\alpha}$ is the geometric realization of the twisted Orlov lattice. More precisely, its first integral cohomology, together with its Hodge structure, is identified with $\bigl(\widehat{\rH}(X,\alpha,\ZZ),\cJ_{\alpha}\bigr)$; see \cite[Proposition 6.3]{LLT25}.
The canonical pairing on the twisted Orlov lattice determines an isomorphism 
\[
\eta_{X,\alpha} \colon A_{X,\alpha} \longrightarrow \widehat{A}_{X,\alpha}.
\]
We call an isomorphism $f \colon A_{X,\alpha} \to A_{X',\alpha'}$ \emph{symplectic} if $\widehat{f} \circ \eta_{X',\alpha'} \circ f = \eta_{X,\alpha}$.

The following result, proved in \cite[Theorem 1.1]{LLT25}, characterizes twisted derived equivalences in terms of these symplectic abelian varieties and thereby confirms the conjecture of Kapustin and Orlov \cite[Remark 2.8]{KO}.

\begin{theorem}\label{thm:de-orlov}
There exists a derived equivalence $\Db(X,\alpha)\simeq\Db(X',\alpha')$ if and only if there is a symplectic isomorphism
\[
f \colon A_{X,\alpha}\xlongrightarrow{\sim} A_{X',\alpha'}.
\]
In particular, any autoequivalence of $\Db(X,\alpha)$ gives rise to a symplectic automorphism of $A_{X,\alpha}$.
\end{theorem}

To see how such symplectic isomorphisms arise from Fourier--Mukai transforms, recall the untwisted case first. Let $X,X'$ be abelian surfaces. Given a derived equivalence $\Psi \colon \Db(X)\to\Db(X')$, Orlov constructed an induced symplectic isomorphism (\cf \cite[Theorem 2.10]{Orlov})
\[
f_\Psi\colon X\times\widehat{X} \xlongrightarrow{\sim} X'\times\widehat{X'}.
\]
The assignment $\Psi \mapsto f_\Psi$ is a group homomorphism.  Its image
\[
\rU(X\times\widehat{X},X'\times\widehat{X'}) \coloneqq \operatorname{im}\bigl(\operatorname{Eq}(\Db(X),\Db(X'))\to \operatorname{Iso}(X\times\widehat{X},X'\times\widehat{X'})\bigr)
\]
is the \emph{Mukai-Polishchuk group}.  When $X=X'$, we write it as $\rU(X\times\widehat{X})$.  There is a short exact sequence
\[
0\longrightarrow \mathbb{Z}\times X\times\widehat{X} \longrightarrow \Aut(\Db(X)) \longrightarrow \rU(X\times\widehat{X}) \longrightarrow 1,
\]
where the kernel consists of functors of the form $\bL t_a^*(\,\cdot\,)\otimes^\bL D [m]$ with $a\in X$, $D\in\Pic^0(X)$ and $m\in\mathbb{Z}$.

For twisted abelian surfaces, a parallel construction exists using equivariant Fourier--Mukai kernels \cite{LLT25}. For each $\gamma = [(u,\nu)] \in A_{X,\alpha}$, we can associate a natural equivariant Fourier--Mukai transform as in the untwisted case.
\begin{itemize}
    \item For any $u \in X$, the compatibility $t_{u*}t_g^* = t_g^*t_{u*}$ naturally holds, and there is an induced autoequivalence
    \[
    t_{u*} \colon \Db(X)_{G,\rho,a} \longrightarrow \Db(X)_{G,\rho,a}.
    \]
    \item For any $\nu \in \widehat{X}$, the corresponding line bundle $\sP_\nu$ consists of the data $(\sP_\zeta, \chi_\nu)$, where $\zeta = \mathrm{r}_{\widehat{X}}(\nu)$ and $\chi_\nu(g)\colon t_g^*\sP_\zeta \to \sP_\zeta$ are descent morphisms satisfying certain compatibility conditions \cite[\S 4.2]{LLT25}. This yields an autoequivalence
    \[
    -\otimes (\sP_\zeta, \chi_\nu) \colon \Db(X)_{G,\rho,a} \longrightarrow \Db(X)_{G,\rho,a}.
    \]
\end{itemize}
It is shown in \cite[\S 4.2]{LLT25} that the Fourier--Mukai kernel of the functor
\[
\Phi_{(u,\nu)} \coloneqq t_{u*}(-) \otimes (\sP_\zeta, \chi_\nu) \colon \Db(X)_{G,\rho,a} \longrightarrow \Db(X)_{G,\rho,a}
\]
is a $(G\times G)$-linearised object.
\begin{definition}
For an element $\gamma = [(u,\nu)]\in A_{X,\alpha}$, define
\[
\Phi_\gamma \in \Aut\bigl(\Db(X,\alpha)\bigr)
\]
by the commutative diagram
\[
\begin{tikzcd}[sep=large]
\Db(X,\alpha) \ar[r,"\Phi_\gamma"] \ar[d,"\simeq"] & \Db(X,\alpha) \ar[d,"\simeq"] \\
\Db(X)_{G,\rho,a} \ar[r,"\Phi_{(u,\nu)}"] & \Db(X)_{G,\rho,a},
\end{tikzcd}
\]
where the vertical equivalences are those induced by the \'etale covering $\mathbf{r}_X$. By \cite[Proposition~4.5]{LLT25}, $\Phi_\gamma$ is independent of the chosen representative $(u,\nu)$, and thus is well-defined.
\end{definition}

\begin{definition}
For twisted abelian surfaces $(X,\alpha)$ and $(X',\alpha')$, define
\[
\rU\bigl(A_{X,\alpha},\,A_{X',\alpha'}\bigr) := \operatorname{im}\Bigl( \operatorname{Eq}\bigl(\Db(X,\alpha),\,\Db(X',\alpha')\bigr) \longrightarrow \operatorname{Iso}\bigl(A_{X,\alpha},\,A_{X',\alpha'}\bigr) \Bigr),
\]
where the map is given by \Cref{thm:de-orlov}.
When $(X,\alpha)=(X',\alpha')$, we set $\rU(A_{X,\alpha}) := \rU(A_{X,\alpha},A_{X,\alpha})$ and call it the \emph{twisted Mukai-Polishchuk group} of $(X,\alpha)$.
\end{definition}

By \cite[Proposition 4.5 \& Proposition 6.1]{LLT25}, there is an injective homomorphism
\[
A_{X,\alpha} \hookrightarrow \Aut\bigl(\Db(X,\alpha)\bigr), \quad \gamma \longmapsto \Phi_\gamma,
\]
and the kernel of the map $\Aut(\Db(X,\alpha)) \to \rU(A_{X,\alpha})$ is generated by the shift functor $[1]$ together with all autoequivalences $\Phi_{\gamma}$ for $\gamma \in A_{X,\alpha}$. We summarise their results as follows.

\begin{proposition}\label{prop:twisted-MP}
There is a short exact sequence
\[
0 \longrightarrow \mathbb{Z}\times A_{X,\alpha} \longrightarrow
\Aut\bigl(\Db(X,\alpha)\bigr) \longrightarrow 
\rU(A_{X,\alpha}) \longrightarrow 0.
\]
Consequently, the subgroup $\Aut_0\bigl(\Db(X,\alpha)\bigr)$ of autoequivalences acting trivially on cohomology is isomorphic to $2\mathbb{Z}\times A_{X,\alpha}$.
\end{proposition}

\subsection{Admissible isometries} \label{subsection:admissible}
We now study the cohomological actions of twisted derived equivalences. For K3 surfaces, a Hodge isometry of the Mukai lattice is the natural input for the derived Torelli theorem. 
For abelian surfaces, this condition alone is not sufficient. We therefore introduce an admissibility condition, generalizing Shioda’s notion \cite{Shioda}.

\begin{definition}
Let $\{v_1,\dots,v_4\}$ be a basis of $\rH^1(X,\mathbb{Q})$.  For any ordered subset $I\subseteq\{1<2<3<4\}$, set
\[
v_I := \bigwedge_{i\in I} v_i,\qquad\text{with } v_\emptyset = -1.
\]
The collection $\Bigl\{ v_I \mid |I| = i \Bigr\}$ forms a basis of $\rH^{i}(X,\mathbb{Q})$; consequently, the collections
\[
\Bigl\{ v_I \mid |I| \text{ is even} \Bigr\} \quad\text{and}\quad \Bigl\{ v_I \mid |I| \text{ is odd} \Bigr\}
\]
give bases of $\rH^{\rm even}(X,\mathbb{Q})$ and $\rH^{\rm odd}(X,\mathbb{Q})$, respectively.

The basis $\{v_1,\dots,v_4\}$ is called \emph{admissible} (in the sense of Shioda) if
\[
\int_X v_1\wedge v_2\wedge v_3\wedge v_4 = 1.
\]
When this holds, the induced bases of $\rH^i(X,\QQ)$, $\rH^{\rm even}(X,\mathbb{Q})$ and $\rH^{\rm odd}(X,\mathbb{Q})$ are called \emph{admissible bases}.  
\end{definition}

In terms of an admissible basis $\{v_I\}$, the Mukai pairing on $\rH^{\rm even}(X,\QQ)$ and the Mukai-Polishchuk pairing on $\rH^{\rm odd}(X,\QQ)$ are computed by
\begin{equation}
\Bigl\langle \sum_{I} a_I v_I,\; \sum_{J} b_J v_J \Bigr\rangle = \int_X \left(\sum_{I} a_I v_I\right)\wedge\left(\sum_{J} b_J v_J\right),
\end{equation}
where the sums run over all $I,J$ of appropriate parity. Notice that the sign convention $v_{\emptyset} = -1$ is used here to match the Mukai pairing.

\begin{definition}
Let $X$ and $X'$ be two abelian surfaces.
\begin{itemize}
    \item An isometry $\varphi: \rH^{2}(X,\mathbb{Q}) \xrightarrow{\sim} \rH^{2}(X',\mathbb{Q})$ with respect to the cup product pairing is called \emph{admissible} if it preserves admissible bases (the image of an admissible basis is still an admissible basis). 
    \item An isometry $\varphi: \rH^{\rm even}(X,\mathbb{Q}) \xrightarrow{\sim} \rH^{\rm even}(X',\mathbb{Q})$ with respect to the Mukai pairing is called \emph{admissible} if it preserves admissible bases.  
    \item An isometry $\varphi: \rH^{\rm odd}(X,\mathbb{Q}) \xrightarrow{\sim} \rH^{\rm odd}(X',\mathbb{Q})$ with respect to the Mukai-Polishchuk pairing is called \emph{admissible} if it preserves admissible bases.  
    \item An isomorphism $\varphi: \rH^{\ast}(X,\mathbb{Q}) \xrightarrow{\sim} \rH^{\ast}(X',\mathbb{Q})$ is called \emph{admissible} if it preserves admissible bases. 
\end{itemize}
\end{definition}

\begin{example}\label{example:B-fields}
Tensoring by a line bundle $L$ acts on the cohomology ring by 
\[
x \longmapsto \exp(\ell)x, \quad \ell = c_1(L) \in \rH^2(X,\ZZ).
\]
With respect to any admissible basis of $\rH^{\even}(X,\QQ)$ or $\rH^{\odd}(X,\QQ)$ ordered by degree, multiplication by $\exp(\ell)$ is represented by a lower triangular matrix whose diagonal entries are all equal to $1$. Hence this action is admissible.

This implies that for different choices of $\mathbf{B}$-field lifts $B$ and $B'$ of a Brauer class, the isometry  
\begin{equation}
    \exp(B-B') \colon \widetilde{\rH}(X,B,\ZZ)\to \widetilde{\rH}(X,B',\ZZ)
\end{equation}
is admissible. Hence up to admissible Hodge isometry, $\widetilde{\rH}(X,\alpha,\ZZ)$ is uniquely defined. 
\end{example}

\begin{example}
Let $f\colon X \to Y$ be an isomorphism between abelian surfaces. By Shioda’s Torelli theorem \cite{Shioda}, the induced Hodge isometry $f^* \colon \rH^2(Y,\QQ) \to \rH^2(X,\QQ)$ is admissible in Shioda’s sense. On the hyperbolic subspace generated by $\rH^0(Y,\QQ)$ and $\rH^4(Y,\QQ)$, the action of $f$ is identity. Since the admissibility on the Mukai lattice is, by definition, equivalent to the admissibility on $\rH^2$ together with the orientation-preserving condition on the hyperbolic summand, the isometry $f^* \colon \widetilde{\rH}(Y,\QQ) \to \widetilde{\rH}(X,\QQ)$ is admissible.
\end{example}

\begin{example}[Poincar\'e bundle]
Let $\{v_i\}$ be an admissible basis. The cohomological Fourier--Mukai transform associated with the Poincar\'e bundle $\cP$ gives an isomorphism
\[
\rH^i(X,\QQ) \xlongrightarrow{\sim} \rH^{4-i}(\widehat{X},\QQ) = \rH^{4-i}(X,\QQ)^{\vee}, \quad \forall\ 0\le i\le 4.
\]
The dual basis $v_I^*$ with respect to the Poincar\'e pairing is therefore identified with $\varepsilon_J \cdot v_J$, where $I \cup J = \{1,2,3,4\}$, and the sign $\varepsilon_J$ is determined by
\[
\varepsilon_J = \varepsilon_J \cdot v_I^*(v_I) = \int_X v_J \wedge v_I.
\]
Calculation in \cite[Example 4.2.3]{LZ25} implies that the induced action of the Poincar\'e line bundle 
\[
\rH^2(X,\QQ) \longrightarrow \rH^2(\widehat{X},\QQ) \simeq \rH^2(X,\QQ)^{\vee}
\]
is not admissible. However, it acts as a reflection on the hyperbolic subspace generated by $\rH^0(X,\QQ)$ and $\rH^4(X,\QQ)$, and therefore extends to an admissible isometry on the Mukai lattice.
\end{example}

\subsection{Admissibility for twisted equivalences}

Let $V$ be a free abelian group of rank $4$. By fixing an admissible basis, we identify $\rH^1(X,\mathbb{Q}) \cong V_{\mathbb{Q}}$ for any abelian surface $X$. By definition, we have the following identifications:
\begin{itemize}
    \item The total cohomology decomposes as $\rH^*(X,\mathbb{Q}) \cong \bigwedge^\ast V_{\mathbb{Q}}$.
    \item The Mukai lattice (even part) and the Orlov lattice (odd part) are identified with $\bigwedge^{\mathrm{even}} V_{\mathbb{Q}}$ and $\bigwedge^{\mathrm{odd}} V_{\mathbb{Q}}$, respectively.
\end{itemize}
We call such an identification an \emph{admissible marking}. 
Given another abelian surface $X'$ with an admissible marking $\rH^{1}(X',\mathbb{Q}) \cong V_{\mathbb{Q}}$, 
the admissible markings allow us to identify the cohomologies of $X$ and $X'$ with the wedge product of $V_{\QQ}$. Via these identifications, the relevant sets of isomorphisms become subgroups of the corresponding linear groups:
\begin{itemize}
    \item Under the Mukai--Polishchuk pairing,  
    $\mathrm{U}\bigl(X\times\widehat{X}, X'\times\widehat{X'}\bigr)$
    is identified  as a subgroup  of $\mathrm{SO}(V_{\mathbb{Q}}\oplus V_{\mathbb{Q}}^{\vee})$ (\cf \cite[Proposition~4.3.7]{GLO}); see also \Cref{thm:admissibility}.
    \item The set of algebra isomorphisms 
    $\operatorname{Iso}\bigl(\rH^{*}(X,\mathbb{Q}),\rH^{*}(X',\mathbb{Q})\bigr)$
    is identified with a subgroup of $\operatorname{GL}\bigl(\bigwedge^\ast V_{\mathbb{Q}}\bigr)$.
\end{itemize}

\begin{remark}
Under the identification provided by a relative admissible basis, 
the induced cohomological action of a relative Fourier--Mukai kernel is given by a locally constant matrix; 
consequently, admissibility is invariant under deformation.
\end{remark}

The following theorem is the key technical input connecting derived equivalences with
admissible Hodge isometries.

\begin{theorem} \label{thm:admissibility}
Let $\Psi$ be a derived (anti-)equivalence between $\Db(X,\alpha)$ and $\Db(X',\alpha')$. The induced isometry 
\[
\Psi^{\widetilde{\rH}}:\widetilde{\rH}(X,\alpha,\ZZ)\to \widetilde{\rH}(X',\alpha',\ZZ)
\]
is admissible and orientation-preserving for equivalences, orientation-reversing for anti-equivalences.
\end{theorem}

\begin{proof}
We may first assume $\Psi$ is a derived equivalence; the anti-equivalence case follows by composing with the duality functor.
\setcounter{proofstep}{0}

\proofstep[] By \Cref{example:B-fields}, the admissibility of $\Psi^{\widetilde{\rH}}$ does not depend on the choice of \textbf{B}-fields on both sides. Thus, without further specifying the \textbf{B}-field lifts, we may keep the notation $\widetilde{\rH}(X,\alpha,\ZZ)$ and $\widetilde{\rH}(X',\alpha',\ZZ)$.

\proofstep[] Up to a shift, the Fourier--Mukai kernel of $\Psi$ is a twisted sheaf on $X\times X'$ with respect to the Brauer class $\alpha^{-1}\boxtimes\alpha'$. For untwisted varieties this is a classical result of Orlov \cite[Proposition 3.2]{Orlov}. The twisted case is proved in \cite[Lemma 6.2.4]{Li2025}. Consequently, there exists an ample line bundle $H$ on $X$ and a primitive Mukai vector $\mathbf{v}= (\Psi^{\widetilde{\rH}})^{-1}(0,0,1)\in \widetilde{\rH}(X,\alpha,\ZZ)$ such that
\[
X' \cong M_H(X,\alpha,\mathbf{v}),\qquad \alpha' = -\theta,
\]
where $\theta\in\Br(Y)$ is the Brauer class induced by the moduli space structure (via the universal family), and $\Psi$ is induced (up to shift) by the universal $\alpha^{-1}\boxtimes \alpha'$-twisted sheaf on $X\times X'$.

\proofstep[Untwisted case]
Assume $\alpha=0$ and $\alpha'=0$. This is essentially proved in \cite{GLO} (see also \cite[Corollary 9.61]{HuybrechtsBook}).  Golyshev--Lunts--Orlov constructed a commutative diagram
\[\begin{tikzcd}[column sep=small]
	& {\operatorname{Im}(\rho_{X,X'})} & {\mathrm{Iso}(\rH^\ast(X,\QQ),\rH^{\ast}(X',\QQ))} \\
	{\operatorname{Eq}\bigl(\Db(X),\Db(X')\bigr)} \\
	& {\mathrm{U}(X\times \widehat{X}, X'\times \widehat{X'})}
	\arrow[phantom, "\subset", from=1-2, to=1-3]
	\arrow["{\lambda_{X,X'}}"', from=1-2, to=3-2]
	\arrow["{\rho_{X,X'}}", from=2-1, to=1-2]
	\arrow["{\pi_{X,X'}}"', from=2-1, to=3-2]
\end{tikzcd}\]
where
\[
\begin{aligned}
\pi_{X,X'} &: \operatorname{Eq}(\Db(X),\Db(X')) \longrightarrow \mathrm{U}(X\times \widehat{X}, X'\times \widehat{X'}),\quad \Psi\longmapsto f_\Psi,\\
\rho_{X,X'} &: \operatorname{Eq}(\Db(X),\Db(X')) \longrightarrow \operatorname{Iso}\bigl(\rH^*(X,\mathbb{Q}),\rH^*(X',\mathbb{Q})\bigr),\quad \Psi\longmapsto \Psi^{\rH}.
\end{aligned}
\]
and $\lambda_{X,X'}$ is obtained via the restriction map to the odd part. 

By taking admissible markings, we identify $\rH^1(X,\mathbb{Q}) \cong V_{\mathbb{Q}}$ and $\rH^1(X',\mathbb{Q}) \cong V_{\mathbb{Q}}$, where $V_{\mathbb{Q}}$ is a fixed vector space of dimension $4$. By \cite[Proposition 4.3.7]{GLO}, the image $\operatorname{Im}(\rho_{X,X'})$ is then identified as a subgroup of the spin group $\mathrm{Spin}(V_{\mathbb{Q}}\oplus V_{\mathbb{Q}}^\vee)$, and the map $\lambda_{X,X'}$ is identified as a group homomorphism factoring through
    \[
    \mathrm{SO}(V_{\mathbb{Q}}\oplus V_{\mathbb{Q}}^\vee) \xleftarrow{\text{double cover}} \mathrm{Spin}(V_{\mathbb{Q}}\oplus V_{\mathbb{Q}}^\vee).
    \]
Consequently, the induced isometry $\Psi^{\rH^\ast}$ on the whole cohomology, as well as its restriction to $\rH^{\mathrm{odd}}$, are both admissible (their determinant is $1$ after fixing an admissible marking). In particular, $\Psi^{\widetilde{\rH}}=\Psi^{\rH^{\mathrm{even}}}$ is admissible as well. 

For the orientation, by Orlov's description of equivalences of derived categories of abelian varieties, the induced action on cohomology is obtained from the spin representation of the corresponding symplectic isomorphism. Hence the cohomological transform lies in $\SO^+$ and preserves the canonical orientation of the positive four-space.
This completes the proof for the untwisted case.

\proofstep[General case]
Now let $(X,\alpha)$ and $(X',\alpha')$ be arbitrary twisted abelian surfaces.
One could follow the approach of \cite{GLO} to compute the induced isometry directly using the equivalences constructed in \cite{LLT25}, but this would involve a complicated treatment of the functors induced by inflation. Instead, we give a much simpler proof by deforming the equivalence $\Psi$ to the untwisted case.

Choose a \textbf{B}-field lift $B_X\in\rH^2(X,\QQ)$ of $\alpha$. There exists a connected smooth curve $T$ and a family of polarized abelian surfaces 
\[
\pi:(\mathcal{X},\mathcal{H})\to T
\]
with a distinguished point $t_1$ such that $(\mathcal{X}_{t_1},\mathcal{H}_{t_1})=(X,H)$, together with a family of \textbf{B}-fields $\mathcal{B}-\in \bR^2\pi_*\QQ$ and a relative primitive Mukai vector $\mathbf{v}$. Moreover, there is a point $t_0\in T$ with the fiber $(\mathcal{X}_{t_0},[\mathcal{B}_{t_0}])$ untwisted, that is, $\mathcal{B}_{t_0}\in\rH^{1,1}(\mathcal{X}_{t_0},\QQ)$.

Consider the relative moduli space $\mathcal{M}\to T$ of $\mathcal{H}$-stable twisted sheaves, whose fiber over $t\in T$ is $\mathcal{M}_{\mathcal{H}_t}(\mathcal{X}_t,[\mathcal{B}_t],\mathbf{v}_t)$. The relative universal twisted sheaf induces a derived equivalence between $(\mathcal{X}_t,[\mathcal{B}_t])$ and $(\mathcal{M}_t,\beta_t)$ for some Brauer class $\beta_t$.
\begin{itemize}
    \item If $\mathcal{M}_{t_1}$ is a fine moduli space, then at $t_0$ we obtain a derived equivalence $\Psi_{t_0}$ between two untwisted abelian surfaces $\mathcal{X}_{t_0}$ and $\mathcal{M}_{t_0}$. By Step 2, $\Psi_{t_0}^{\widetilde{\rH}}$ is admissible. Since admissibility is deformation invariant,  $\Psi_{t_1}^{\widetilde{\rH}}$ is also admissible.
    \item If $\mathcal{M}_{t_1}$ is not fine, we apply an additional deformation to the $X' =\mathcal{M}_{t_1}$ side to specialize it to an untwisted abelian surface. Repeating the same argument then yields the admissibility.
\end{itemize}
The orientation argument follows because the relative universal twisted sheaf induces an isometry of the corresponding Mukai local system $\bR^0\pi_*\ZZ \oplus \bR^2\pi_*\ZZ \oplus \bR^4\pi_*\ZZ$ over $T$. Hence the associated spinor norm, viewed as a function on $T$, is locally constant; since $T$ is connected and the values lie in $\{\pm1\}$, it is constant.
\end{proof}

\begin{remark}
One may ask whether the diagram above admits an integral refinement. More precisely, after choosing compatible $\mathbf{B}$-field twists, one can ask whether the cohomological Fourier--Mukai transform identifies the corresponding integral structures, namely the twisted Mukai lattice and the twisted Orlov lattice, inside the cohomology ring.
\end{remark}

\subsection{A derived Torelli theorem via the twisted Mukai lattice}

The previous discussion on admissibility and the action of derived equivalences on cohomology naturally leads to a Torelli-type statement in terms of the twisted Mukai lattice. An analogous result for K3 surfaces and twisted K3 surfaces was obtained in \cite{HS05,Reinecke19}.  The overall strategy parallels the K3 case, yet certain subtleties demand much closer attention. 

\begin{theorem}[\Cref{thm:torelli-C}, derived Torelli theorem] \label{thm:derived-torelli}
Let $(X,\alpha)$ and $(X',\alpha')$ be twisted complex abelian surfaces.  There exists an admissible orientation-preserving (resp. orientation-reversing) Hodge isometry
\[
h: \widetilde{\rH}(X,\alpha,\mathbb{Z}) \xrightarrow{\sim} \widetilde{\rH}(X',\alpha',\mathbb{Z})
\]
if and only if there exists a derived equivalence (resp. anti-equivalence)
\[
\Psi: \Db(X,\alpha) \xrightarrow{\sim} \Db(X',\alpha') \quad (\text{resp. }\Psi \colon \Db(X,\alpha) \xrightarrow{\sim} \Db(X',\alpha')^{\mathrm op})
\]
such that $\Psi$ induces $h$ on the twisted Mukai lattices.
\end{theorem}

\begin{proof}

It suffices to deal with the equivalence case. The anti-equivalence case follows by composing with the derived duality functor, which identifies $\Db(X',\alpha')^{\mathrm op}$ with $\Db(X',-\alpha')$.

\textit{Necessity.} If $\Psi$ is a derived equivalence, then it induces a symplectic isomorphism in $\rU(A_{X,\alpha}, A_{X',\alpha'})$.  Its action on odd cohomology is admissible.  Moreover, its action $\Psi^{H^\ast}\in \mathrm{Im}(\rho)$ is also admissible. It follows that $\Psi^{\widetilde{\rH}}$ is admissible as well. 

\textit{Sufficiency.} Suppose we are given an admissible orientation-preserving Hodge isometry
\[
h: \widetilde{\rH}(X,\alpha,\mathbb{Z}) \xrightarrow{\sim} \widetilde{\rH}(X',\alpha',\mathbb{Z}).
\]
Choose \textbf{B}-field lifts $B_X\in \rH^2(X,\mathbb{Q})$ and $B_{X'}\in \rH^2(X',\mathbb{Q})$ for the Brauer classes $\alpha$ and $\alpha'$, so that
\[
\widetilde{\rH}(X,\alpha,\mathbb{Z}) = \exp(B_X)\,\widetilde{\rH}(X,\mathbb{Z}),\qquad
\widetilde{\rH}(X',\alpha',\mathbb{Z}) = \exp(B_{X'})\,\widetilde{\rH}(X',\mathbb{Z}).
\]
Define the Mukai vector
\[
\mathbf{v} := \exp(-B_{X'})\Bigl(h\bigl(\exp(B_X)(0,0,1)\bigr)\Bigr)\;\in\; \widetilde{\rH}(X',\mathbb{Z}).
\]
Because $h$ is an isometry of the twisted lattices and $\exp(B_X)(0,0,1)$ is a primitive isotropic vector in $\widetilde{\rH}(X,\alpha,\mathbb{Z})$, the vector $\mathbf{v}$ is primitive and isotropic as well.

Choose a $\mathbf{v}$-generic stability condition $\sigma$, and let $M = M_{\sigma}(X',\alpha',\mathbf{v})$ denote the moduli space of $\sigma$-stable objects on $(X',\alpha')$ with Mukai vector $\mathbf{v}$.  It is easy to see that $M$ is an abelian surface.  Let $(\sM,\beta_M)\to M$ be the $\mathbb{G}_m$-gerbe naturally associated to the universal family. The universal object induces a Fourier--Mukai equivalence
\[
\Psi_{\mathcal{E}} : \Db(M,\beta_M) \xrightarrow{\sim} \Db(X',\alpha')
\]
whose kernel is a $(1,1)$-twisted perfect complex.  This transform gives a Hodge isometry
\[
g: \widetilde{\rH}(M,\beta_M,\mathbb{Z}) \xrightarrow{\sim} \widetilde{\rH}(X',\alpha',\mathbb{Z}).
\]

Set
\[
\eta \coloneqq g^{-1}\circ h : \widetilde{\rH}(X,\alpha,\mathbb{Z}) \longrightarrow \widetilde{\rH}(M,\beta_M,\mathbb{Z}).
\]
By construction $\eta$ is a Hodge isometry and $\eta(0,0,1) = (0,0,1)$ (after adjusting B-field lifts appropriately).  Consequently, the image of $(1,0,0)$ must be of the form
\[
\eta(1,0,0) = \bigl(1,\; b,\; \tfrac{b^2}{2}\bigr)
\]
for some $b\in \rH^2(M,\mathbb{Q})$.  Replace the B-field lift on $M$ by $B_M' := B_M - b$.  After this replacement, $\eta$ additionally fixes $(1,0,0)$ and therefore preserves the hyperbolic sublattice generated by $(1,0,0)$ and $(0,0,1)$.  Hence $\eta$ restricts to a Hodge isometry on the orthogonal complement:
\[
\eta_2 \coloneqq \eta\big|_{\rH^2(X,\mathbb{Z})} : \rH^2(X,\mathbb{Z}) \xrightarrow{\sim} \rH^2(M,\mathbb{Z}),
\]
using the identification $(0,0,1)^\perp/(0,0,1) \cong \rH^2$.

Because both $g$ and $h$ are orientation-preserving, $\eta_2$ is also orientation-preserving.  Moreover, in terms of admissible bases, $\eta_2$ is admissible.  By Shioda's Torelli theorem for abelian surfaces \cite{Shioda}, there exists an isomorphism
\[
f \colon X \xlongrightarrow{\sim} M
\]
such that $\eta_2 = (f^*)^{-1}$.  The isomorphism $f$ lifts to an equivalence of twisted derived categories
\[
\bL f^* \colon \Db(M,\beta_M) \xrightarrow{\sim} \Db(X,\alpha)
\]
after identifying the gerbes via the B-field shift; indeed the Hodge isometry $\eta$ guarantees that $f^*\beta_M$ is equivalent to $\alpha$ as Brauer classes.

Finally, the composition
\[
\Psi := \Psi_{\mathcal{E}} \circ (\bL f^*)^{-1} : \Db(X,\alpha) \longrightarrow \Db(X',\alpha')
\]
is a derived equivalence.  By construction its induced isometry on twisted Mukai lattices coincides with $h$ (up to a possible sign, which can be absorbed by a shift).  This completes the proof.
\end{proof}

\section{Descent of autoequivalences to twisted Kummer surfaces} \label{sec:kummer-descent}

In this section, we develop a descent theory for autoequivalences from twisted abelian surfaces to the associated twisted Kummer surfaces. This is the bridge that allows us to reduce the Bloch-type statement for twisted abelian surfaces to the corresponding statement for twisted K3 surfaces.

\subsection{Twisted Kummer surfaces}\label{subsection:kummer}

Let \(X\) be an abelian surface and let \(\iota\colon X\to X\) denote the \((-1)\)-involution.  
The quotient surface \(X/\langle \iota \rangle\) has \(16\) ordinary double points, corresponding to the fixed points \(X[2]\). Let
\[
p \colon \operatorname{Km}(X) \longrightarrow X/\langle\iota\rangle
\]
be the minimal (crepant) resolution. The resulting smooth surface \(\km(X)\) is the Kummer surface associated with \(X\). Alternatively, one can first blow up \(X\) at the \(16\) two-torsion points to obtain \(\widetilde{X}\), and then take the quotient by the lifted involution of $\iota$ on $\widetilde{X}$. The following commutative diagram summarises the two constructions:
\[
\begin{tikzcd}
& \widetilde{X} \ar[dl,"\pi"] \ar[dr,"\kappa"'] & \\
X \ar[dr,"q"] & & \km(X) \ar[dl,"p"'] \\
& X/\langle\iota\rangle &
\end{tikzcd}
\]
Here \(\widetilde{X}\) is the blow-up of \(X\) along \(X[2]\), \(\pi\colon\widetilde{X}\to X\) is the blow-up, and \(\kappa \colon \widetilde{X} \to \km(X)\) is the quotient by the lifted involution.

Let \(E \subset \widetilde{X}\) be the exceptional locus of \(\pi\). By Grothendieck's purity theorem (see e.g. \cite[Theorem 1.1]{Ces19}), the restriction maps induce isomorphisms of Brauer groups:
\[
\Br(X) \simeq \Br(X\setminus X[2]) = \Br(\widetilde{X}\setminus E) \simeq \Br(\widetilde{X}).
\] 
Moreover, Skorobogatov and Zarhin \cite[Proposition 2.7]{SZ17} proved that \(\kappa^* : \Br(\km(X)) \to \Br(\widetilde{X})\) is an isomorphism. Consequently, we obtain

\begin{lemma} \label{lem:brauer-iso}
There is a canonical isomorphism
\[
\Br(X) \xrightarrow{\sim} \Br(\km(X))
\]
such that, for any Brauer class \(\alpha \in \Br(X)\), the corresponding class \(\bar{\alpha} \in \Br(\km(X))\) satisfies
\[
\pi^*\alpha = \kappa^*\bar{\alpha} \quad \text{in } \Br(\widetilde{X}).
\] 
\end{lemma}

We denote the twisted Kummer surface associated to \((X,\alpha)\) by \((\km(X), \bar{\alpha})\) or simply \(\km(X,\alpha)\).

\subsection{The twisted BKR equivalence for Kummer surfaces}
Let \((X,\alpha)\) be a twisted abelian surface. Set \(I = \mathbb{Z}/2\mathbb{Z} = \{1,g\}\). We define a {left \(I\)-action} \((\rho,\sigma)\) on the bounded derived category \(\mathrm{D}^{\mathrm{b}}(X,\alpha)\) as follows:
\begin{itemize}
    \item Since $\iota^*$ acts as identity on the Brauer group, choose an identification \(\iota^*\alpha \simeq \alpha\).
    \item  
    \(\rho_g \coloneqq \mathbf{L}\iota^\ast : \mathrm{D}^{\mathrm{b}}(X,\alpha) \to \mathrm{D}^{\mathrm{b}}(X,\alpha)\) and $\rho_1=\id$.  
    \item We fix a natural isomorphism \(\sigma_{g,g} : \rho_g \circ \rho_g \to \rho_1\) and set \(\sigma_{1,g} = \sigma_{g,1} = \operatorname{id}\).  
    These data satisfy the cocycle condition  
    \[
    \sigma_{g,g} \circ (\rho_g(\sigma_{g,g})) = \sigma_{g^2,g} \circ \sigma_{g,g^2}.
    \]
\end{itemize}
We thus obtain an equivariant category $\mathrm{D}^{\mathrm{b}}(X,\alpha)_{I,\rho,\sigma}$ as in \Cref{subsection:equivariant-category}. For simplicity, we set  
\[
\mathrm{D}^{\mathrm{b}}(X,\alpha)_I \coloneqq \mathrm{D}^{\mathrm{b}}(X,\alpha)_{I,\rho,\sigma}.
\]  
This category is equipped with the forgetful functor  
\[
\mathtt{For} : \mathrm{D}^{\mathrm{b}}(X,\alpha)_I \longrightarrow \mathrm{D}^{\mathrm{b}}(X,\alpha).
\]  

Now consider the universal \(I\)-cluster  
\[
Z := \{(y,x) \in \operatorname{Km}(X) \times X \mid x \text{ lies in the }I\text{-cluster corresponding to }y\}.
\]  
The space \(Z\) is naturally isomorphic to \(\widetilde{X}\) via the map  
\[
(\kappa,\pi) : \widetilde{X} \longrightarrow \operatorname{Km}(X) \times X.
\]  
By construction, $\Db(\widetilde{X},\pi^\ast\alpha)$ admits a natural $I$-action, with respect to which $\bL\pi^*$ is equivariant. We write $\bL\pi_I^* \dashv \bR\pi_{I*}$ for the induced adjoint pair on equivariant derived categories. We have the following twisted version of the BKR equivalence.

\begin{theorem}[Twisted BKR equivalence]\label{thm:twisted-bkr}
 There exists a Fourier--Mukai equivalence
\[
\Xi: \Db\bigl(\km(X),\bar\alpha\bigr) \xrightarrow{\sim} \Db(X,\alpha)_I.
\]
 fitting into the commutative diagram
\begin{equation}\label{eq:kummer-bkr}
\begin{tikzcd}[sep=large]
   \Db(\km(X),\bar\alpha) \ar[r,"\Xi"] \ar[dr,"\bL\kappa^*"] &
   \Db(X,\alpha)_I \ar[r,"\mathtt{For}"] &
   \Db(X,\alpha)  \\
 & \Db(\widetilde{X},\pi^\ast\alpha)_I  \ar[r,"\mathtt{For}"] \ar[u,"\bR\pi_{I*}"]&
   \Db(\widetilde{X},\pi^\ast\alpha) \ar[u,"\bR\pi_*"].
\end{tikzcd}
\end{equation}
\end{theorem}

\begin{proof}
We follow the strategy of the classical BKR equivalence for Kummer surfaces (see \cite{BKR, BottiniHuybrechts}), adapted to the twisted setting.

By definition, the structure sheaf \(\cO_Z\) carries a natural \(I\)-linearization 
\[
\{\theta_g : g^\ast \cO_Z \to \cO_Z\}.
\]
Here $\theta_g$ denotes the automorphism of the structure sheaf induced by the $I$-action on $Z$. This action comes from the diagonal action of $I$ on \(\km(X)\times X\), namely
\[
g\cdot(y,x)=(y,g\cdot x).  
\]
Moreover, \Cref{lem:brauer-iso} gives \(\kappa^*\bar\alpha = \pi^*\alpha\) in \(\Br(\widetilde{X})\), so the class \(\kappa^*\bar\alpha^{-1}\otimes \pi^*\alpha\) is trivial on \(\widetilde{X}\). Thus we can view $\cO_Z$ as a twisted sheaf in $\Db\bigl(\km(X)\times X,\; \bar\alpha^{-1}\boxtimes \alpha\bigr)$. Since $\rH^2(I,\mathbb{C}^*)=0$, the corresponding trivialization may be chosen to be \(I\)-equivariant. With this choice, together with the above $I$-linearization, \(\cO_Z\) becomes a twisted \(I\)-equivariant Fourier--Mukai kernel 
\[
\mathcal{P} \in \Db\bigl(\km(X)\times X,\; \bar\alpha^{-1}\boxtimes \alpha\bigr)_I.
\]  

Define the functor  
\[
\Xi: \Db(\km(X),\bar\alpha) \longrightarrow \Db(X,\alpha)_I
\]  
by  
\[
\Xi(-) := \bR q_* \bigl(\mathbf{L}p^*(-)\otimes \mathcal{P} \bigr) = \mathbf{R}\pi_{I*} \circ \mathbf{L}\kappa^*(-),
\]  
where $p,q$ are the projections from the product to $\km(X)$ and $X$ respectively, and \(\otimes\) denotes the twisted tensor product in the equivariant sense. Here we identify $Z$ with $\widetilde{X}$ for the second equality.
After forgetting the equivariant structure, we obtain 
\begin{equation} \label{eq:BKR-compose-for}
\Theta_{X,\alpha} \coloneqq \mathtt{For}\circ\Xi = \bR q_* \bigl(\mathbf{L}p^*(-)\otimes \cO_Z \bigr) = \mathbf{R}\pi_* \circ \mathbf{L}\kappa^*.
\end{equation}
This is precisely the Fourier--Mukai transform with kernel \(\cO_Z\). This proves the commutativity of diagram \cref{eq:kummer-bkr}.

We now show that $\Xi$ is a derived equivalence. The original BKR argument (\cf \cite{BKR}) gives a Fourier--Mukai equivalence
\[
\Db(\km(X)) \longrightarrow \Db(X)_I
\]
with kernel $\cO_Z \in \Db(\km(X) \times X)_{1\times I}$, equipped with a canonical $(1 \times I)$-linearization. Since the $(1\times I)$-action here is the same as the $I$-action considered above, this canonical linearization agrees with $\{\theta_g \colon g^*\cO_Z \to \cO_Z\}$. Introducing the twist does not change the underlying linearization data.

Thus for our twisted setting, after using the chosen $I$-equivariant trivialization of the Brauer classes, $\Xi$ agrees \'etale locally with the ordinary BKR equivalence. Consequently, the unit and counit of the adjunction for $\Xi$ are isomorphisms \'etale locally. Since being an isomorphism is a local property, they are isomorphisms globally. It follows that $\Xi$ is an equivalence. 
\end{proof}

\subsection{Adjusting autoequivalences for Kummer descent}

This subsection shows that, after composing with a suitable element of $A_{X,\alpha}$, an autoequivalence becomes compatible with the involution, and hence descends to the Kummer surface. In \Cref{sec:bloch-surface} we will show that this adjustment acts trivially on the Albanese kernel.

Recall from \Cref{prop:twisted-MP} the short exact sequence
\[
0 \longrightarrow \mathbb{Z}\times A_{X,\alpha} \longrightarrow
\Aut\bigl(\Db(X,\alpha)\bigr) \longrightarrow 
\rU(A_{X,\alpha}) \longrightarrow 0.
\]
Under the identification $\iota^*\alpha \simeq \alpha$, $\bL\iota^*$ acts on $\Db(X,\alpha)$ as an involution.

\begin{proposition}\label{prop:commute-involution}
Let $\Psi\colon \Db(X,\alpha) \xrightarrow{\sim} \Db(X',\alpha')$ be a derived equivalence.  Then there exists $\gamma_0 \in A_{X,\alpha}$ such that
\[
\bL\iota_{X'}^* \circ \Psi' = \Psi' \circ \bL\iota_X^*,
\]
where $\Psi' \coloneqq \Psi \circ \Phi_{\gamma_0}$, and $\Phi_{\gamma_0}$ is the autoequivalence corresponding to $\gamma_0 \in A_{X,\alpha}$.
\end{proposition}

\begin{proof}
Consider the commutator
\[
\Omega \coloneqq \Psi^{-1} \circ \bL\iota_{X'}^* \circ \Psi \circ (\bL\iota_X^*)^{-1} \in \Aut(\Db(X,\alpha)).
\]
Since $\bL\iota_X^*$ and $\bL\iota_{X'}^*$ map to $-\mathrm{id}$ in $\rU(A_{X,\alpha})$ and $\rU(A_{X',\alpha'})$ respectively, their conjugates cancel. Hence $\Omega$ lies in the kernel $\mathbb{Z}\times A_{X,\alpha}$.  Write $\Omega = \Phi_{\gamma} \circ [k]$ with $\gamma \in A_{X,\alpha}$ and $k \in \mathbb{Z}$.

Because $(\bL\iota_X^*)^2 = \mathrm{id}$, we have $(\Omega \circ \bL\iota_X^*)^2 = \mathrm{id}$. On the other hand, a direct computation gives  
\[
\bL\iota_X^* \circ \Phi_\gamma = \Phi_{-\gamma} \circ \bL\iota_X^*.
\]
This forces $k=0$ and thus $\Omega = \Phi_\gamma$.  Consequently,
\begin{equation} \label{eq:proof-descending}
\Psi^{-1} \circ \bL\iota_{X'}^* \circ \Psi = \Phi_\gamma \circ \bL\iota_X^*.
\end{equation}

Since $A_{X,\alpha}$ is $2$-divisible, choose $\gamma_0 \in A_{X,\alpha}$ such that $2\gamma_0 = \gamma$. We claim that $\Psi' \coloneqq \Psi \circ \Phi_{\gamma_0}$ satisfies the desired commutativity. Indeed, using \cref{eq:proof-descending} and the relation $\bL\iota_X^* \circ \Phi_{\gamma_0} = \Phi_{-\gamma_0} \circ \bL\iota_X^*$, we compute
\[
\begin{aligned}
\bL\iota_{X'}^* \circ \Psi' &= \bL\iota_{X'}^* \circ \Psi \circ \Phi_{\gamma_0} \\
&= \Psi \circ \Phi_\gamma \circ \bL\iota_X^* \circ \Phi_{\gamma_0} \\
&= \Psi \circ \Phi_\gamma \circ \Phi_{-\gamma_0} \circ \bL\iota_X^* \\
&= \Psi \circ \Phi_{\gamma-\gamma_0} \circ \bL\iota_X^*.
\end{aligned}
\]
On the other hand,
\[
\Psi' \circ \bL\iota_X^* = \Psi \circ \Phi_{\gamma_0} \circ \bL\iota_X^*.
\]
Because $2\gamma_0 = \gamma$, we have $\gamma - \gamma_0 = \gamma_0$. Hence $\Phi_{\gamma-\gamma_0} = \Phi_{\gamma_0}$, and so
\[
\bL\iota_{X'}^* \circ \Psi' = \Psi \circ \Phi_{\gamma_0} \circ \bL\iota_X^* = \Psi' \circ \bL\iota_X^*,
\]
which completes the proof.
\end{proof}

Combining the twisted BKR equivalence with the adjustment result above, we obtain the following descent statement for derived equivalences on twisted abelian surfaces to twisted Kummer surfaces.

\begin{proposition}\label{prop:descend-equivalence}
Let $\Psi \colon \Db(X,\alpha) \xrightarrow{\sim} \Db(X',\alpha')$ be a derived equivalence.  Then there exists $\gamma \in A_{X,\alpha}$ such that $\Psi' \coloneqq \Psi \circ \Phi_{\gamma}$ descends to an equivalence between the associated twisted Kummer surfaces, i.e. there exists a commutative diagram:
\begin{equation}\label{eq:des-kummer}
    \begin{tikzcd}[sep=large]
    \Db(\km(X),\bar\alpha) \ar[r,"\overline{\Psi}", "\sim"'] \ar[d,"\Theta_{X,\alpha}"'] & 
    \Db(\km(X'),\overline{\alpha'}) \ar[d,"\Theta_{X',\alpha'}"] \\
    \Db(X,\alpha) \ar[r,"\Psi'", "\sim"'] & 
    \Db(X',\alpha'),
\end{tikzcd}
\end{equation}
where $\Theta_{X,\alpha}$ and $\Theta_{X',\alpha'}$ are the functors given in \cref{eq:BKR-compose-for}.
\end{proposition}

\begin{proof}
First, we show that  $\Psi' \coloneqq \Psi \circ \Phi_{\gamma}$ descends to the associated equivariant categories under the action of $I$. By \Cref{prop:commute-involution}, we may choose $\gamma \in A_{X,\alpha}$ such that $\bL\iota_{X'}^* \circ \Psi' \simeq \Psi' \circ \bL\iota_X^*$. There is a canonical choice of this commutation isomorphism, which is given by the canonical isomorphism of kernels:
\[
\vartheta_g \colon (\iota_X\times \id_{X'})^*\cP' \xlongrightarrow{\sim} (\id_X\times \iota_{X'})^*\cP'.
\]
As $I$ is cyclic, $\rH^2(I,\CC^*)=0$. Hence $\vartheta_g$ satisfies the coherence condition (see \cite[\S 2.1]{BO23}), and therefore $\Psi'$ descends to an equivalence of $I$-equivariant categories:
\[
\Psi'_I \colon \Db(X,\alpha)_I \xlongrightarrow{\sim} \Db(X',\alpha')_I.
\]

Using the twisted BKR equivalences from \Cref{thm:twisted-bkr}, we obtain a derived equivalence
\[
\overline{\Psi}: \Db(\km(X),\bar\alpha) \xlongrightarrow{\Xi_X} \Db(X,\alpha)_I \xlongrightarrow{\Psi'_I} \Db(X',\alpha')_I \xlongrightarrow{\Xi_Y^{-1}} \Db(\km(X'),\overline{\alpha'}).
\]
The resulting functor is therefore a derived equivalence between the associated twisted Kummer surfaces. The commutativity of \cref{eq:des-kummer} follows from the commutative diagram \cref{eq:kummer-bkr} in \Cref{thm:twisted-bkr}.
\end{proof}

\section{Bloch's conjecture for twisted abelian surfaces} \label{sec:bloch-surface}

We first review the action of twisted derived equivalences on Chow groups for K3 surfaces, which serves as the foundation for the abelian surface case via the Kummer construction.

\subsection{The Albanese kernel of twisted K3 surfaces}

Recall that Beauville and Voisin \cite{BV04} established a canonical decomposition for the Chow group of a K3 surface $S$:
\[
\operatorname{CH}_*(S) = R(S) \oplus \operatorname{CH}_0(S)_{\mathrm{alb}},
\]
where $R(S) = \operatorname{CH}_2(S) \oplus \operatorname{CH}_1(S) \oplus \mathbb{Z}\cdot\mathfrak{o}_S$ is the Beauville--Voisin subgroup, and $\mathfrak{o}_S$ is the Beauville--Voisin class. For any derived equivalence $\Psi : \Db(S) \xrightarrow{\sim} \Db(S')$ between K3 surfaces, Huybrechts \cite[Theorem 2]{Huybrechts10} proved that $\Psi^{\CH_*}$ preserves the Beauville--Voisin subgroup, and hence induces a map between $\operatorname{CH}_0(S)_{\mathrm{alb}}$ and $\operatorname{CH}_0(S')_{\mathrm{alb}}$.

For twisted K3 surfaces, it remains open in general whether $\Psi^{\CH_*}$ preserves $R(S)$.  However, the result of \cite{CLZZ:2024} shows that a suitable modification of $\Psi^{\CH_*}$ preserves the Beauville--Voisin subgroup. Translating this modified preservation statement back to the usual Chow action gives the following consequence.

\begin{proposition}[{\cite[Prop. 5.1]{CLZZ:2024}}] \label{prop:twisted-derived-preserves-BV}
Let $\Psi : \Db(S,\gamma) \to \Db(S',\gamma')$ be a derived equivalence between twisted K3 surfaces. Then
\[
\Psi^{\CH_*}\bigl(\operatorname{CH}_0(S)_{\mathrm{alb}}\bigr) = \operatorname{CH}_0(S')_{\mathrm{alb}}.
\]
\end{proposition}

\subsection{Kummer descent and the Albanese kernel of twisted abelian surfaces}

For any abelian variety $X$ of dimension $g$, there is a canonical decomposition
\[
\CH_i(X)=\bigoplus_{s=-i}^{g-i}\CH_i(X)_{(s)},
\]
where $\CH_i(X)_{(s)}$ is the eigenspace for the action of multiplication by $n$ (see \cite{BeauvilleChow}):
\[
\mathbf{n}_{X\ast} (z)= n^{2i+s}z \qquad\text{for }z\in\CH_i(X)_{(s)},
\]
For $0$-cycles, this gives
\[
\CH_0(X)=\CH_0(X)_{(0)}\oplus\CH_0(X)_{(1)}\oplus\CH_0(X)_{(2)}.
\]
Here $\CH_0(X)_{(2)}$ is the Albanese kernel $\CH_0(X)_{\mathrm{alb}}$ (the kernel of the Albanese map), 
$\CH_0(X)_{(0)}$ is generated by a point class (e.g. $\chow{0}$), and $\CH_0(X)_{(1)}$ is a complementary part mapping isomorphically to the Albanese part.

Now let $(X,\alpha)$ be a twisted abelian surface, and let $(\km(X),\bar\alpha)$ be the associated twisted Kummer surface. 
The Kummer construction in \Cref{subsection:kummer} induces an inclusion map
\begin{equation}\label{eq:kummer-chow}
  \pi_\ast \circ \kappa^\ast: \CH_0(\Km(X))\hookrightarrow\CH_0(X).
\end{equation}
which sends $\CH_0(\km(X))_{\rm alb}$ to $\CH_0(X)_{\alb}$.

Recall from \Cref{thm:twisted-bkr} that the twisted BKR equivalence gives a functor
\[
\Theta_{X,\alpha} := \mathtt{For} \circ \Xi = \bR\pi_* \circ\bL\kappa^*: \Db(\km(X),\bar\alpha) \longrightarrow \Db(X,\alpha),
\]
where \(\Xi: \Db(\km(X),\bar\alpha) \to \Db(X,\alpha)_I\) is the twisted BKR equivalence and \(\mathtt{For}: \Db(X,\alpha)_I \to \Db(X,\alpha)\) is the forgetful functor. 
The commutative diagram \cref{eq:kummer-bkr} implies the following compatibility on Chow groups.

\begin{lemma}\label{lem:chowofTheta}
For any twisted abelian surface $(X,\alpha)$, the functor $\Theta_{X,\alpha}$ induces a map on Chow groups
\[
\Theta_{X,\alpha}^{\CH_*}: \CH_*(\km(X)) \longrightarrow \CH_*(X)
\]
which can be identified with the map given in \cref{eq:kummer-chow}. In particular, $\Theta_{X,\alpha}^{\CH_\ast}$ restricts to an isomorphism
\[
\CH_0(\Km(X))_{\alb}\cong\CH_0(X)_{(2)}.
\]
\end{lemma}

For our purposes, the induced action on the Albanese kernel $\CH_0(X)_{\alb}$ is essential, since this is the part relevant to the Bloch-type statement.

\begin{proposition}\label{prop:albanese-kernel-action}
For any $\gamma \in A_{X,\alpha}$, the induced action of $\Phi_\gamma$ on $\CH_0(X)_{\mathrm{alb}}$ is the identity.
\end{proposition}

\begin{proof}
Recall that, under the equivalence $\Db(X,\alpha)\simeq \Db(X)_{G,\rho,a}$, $\Phi_\gamma$ corresponds to the functor 
\[
\Phi_{(u,\nu)} \cong t_{u*}(-) \otimes (\sP_\zeta,\chi_\nu),
\]
where $\gamma = [(u,\nu)]$ and $\zeta = \mathrm{r}_{\widehat{X}}(\nu)$.
Its action on $\CH_0(X)_{\mathrm{alb}}$ is induced by the ordinary Fourier--Mukai transform $t_{u*}(-)\otimes \sP_\zeta$. Thus it suffices to consider the untwisted case.

\begin{itemize}
    \item \textbf{Translation by $u$:} For any zero-cycle $z\in\CH_0(X)_{\mathrm{alb}}$, we have $(t_u)_*(z)=z * \chow{u}$, where $*$ denotes the Pontryagin product on $\CH_0(X)$. Since $z$ lies in the Albanese kernel, one has $z * (\chow{u}-\chow{0}) = 0$ (see e.g. \cite[Lemma 3.2]{BeauvilleChow}). Hence $(t_u)_*(z)=z$.
    \item \textbf{Tensor product with $\sP_\zeta$:} Tensoring with the line bundle $\sP_\zeta$ acts by multiplication by its Chern character $\ch(\sP_\zeta)$. Only the degree-$0$ component of $\ch(\sP_\zeta)$ contributes to $\CH_0(X)$, and that component is $1$. Therefore the action is the identity.
\end{itemize}

Combining the two, $\Phi_{(u,\nu)}$ acts as the identity on $\CH_0(X)_{\mathrm{alb}}$, and the same holds for $\Phi_\gamma$ via the equivalence. 
\end{proof}

We now establish an analogue of \Cref{prop:twisted-derived-preserves-BV} for derived equivalences between twisted abelian surfaces.

\begin{proposition}\label{prop:preserve-albanese-kernel}
Let $(X,\alpha)$ and $(X',\alpha')$ be two twisted abelian surfaces, and let 
\[
\Psi : \Db(X,\alpha) \xrightarrow{\sim} \Db(X',\alpha')
\]
be a derived equivalence. Then we have
\[
\Psi^{\CH_*}\bigl(\operatorname{CH}_0(X)_{\mathrm{alb}}\bigr) = \operatorname{CH}_0(X')_{\mathrm{alb}}.
\]
\end{proposition}

\begin{proof}
First note that any element $\Phi_0 \in \mathbb{Z}\times A_{X,\alpha} \subseteq \operatorname{Aut}(\Db(X,\alpha))$ acts as $\pm\operatorname{id}$ on $\operatorname{CH}_0(X)_{\mathrm{alb}}$. Indeed, such a $\Phi_0$ is a composition of a shift $[k]$ and an element $\Phi_\gamma$ with $\gamma \in A_{X,\alpha}$. The shift acts as $(-1)^k$ on $\operatorname{CH}_0(X)_{\mathrm{alb}}$, while \Cref{prop:albanese-kernel-action} shows that each $\Phi_\gamma$ acts trivially. Hence the total action is $\pm\operatorname{id}$. 

By \Cref{prop:descend-equivalence}, after possibly composing $\Psi$ on the right with such elements (which do not affect the desired inclusion), we may assume that $\Psi$ descends to a derived equivalence $\overline{\Psi} \colon \Db(\km(X),\bar\alpha) \to \Db(\km(X'),\overline{\alpha'})$ between the associated twisted Kummer surfaces. This descent fits into the commutative diagram
\[
\begin{tikzcd}[sep=large]
\Db(\km(X),\bar\alpha) \ar[d,"\Theta_{X,\alpha}"'] \ar[r,"\overline{\Psi}"] & \Db(\km(X'),\overline{\alpha'}). \ar[d,"\Theta_{X',\alpha'}"] \\
\Db(X,\alpha) \ar[r,"\Psi"] & \Db(X',\alpha')
\end{tikzcd}
\]
Applying the induced actions on Chow groups gives a commutative square 
\begin{equation} \label{eq:Chow-diagram}
\begin{tikzcd}[sep=large]
\CH_*(\km(X)) \ar[r,"\overline{\Psi}^{\CH_*}"] \ar[d,"\Theta_{X,\alpha}^{\CH_*}"'] & 
\CH_*(\km(X')) \ar[d,"\Theta_{X',\alpha'}^{\CH_*}"] \\
\CH_*(X) \ar[r,"\Psi^{\CH_*}"] & \CH_*(X').
\end{tikzcd}
\end{equation}
By \Cref{lem:chowofTheta}, $\Theta_{X,\alpha}^{\CH_0}$ and $\Theta_{X',\alpha'}^{\CH_0}$ restrict to isomorphisms
\[
\Theta_{X,\alpha}^{\CH_0}: \CH_0(\km(X))_{\mathrm{alb}} \xrightarrow{\simeq} \CH_0(X)_{\mathrm{alb}},\qquad
\Theta_{X',\alpha'}^{\CH_0}: \CH_0(\km(X'))_{\mathrm{alb}} \xrightarrow{\simeq} \CH_0(X')_{\mathrm{alb}}.
\]
By \Cref{prop:twisted-derived-preserves-BV}, $\overline{\Psi}^{\CH_*}$ sends $\CH_0(\km(X))_{\mathrm{alb}}$ to $\CH_0(\km(X'))_{\mathrm{alb}}$. The commutativity of the diagram then implies the desired inclusion. Applying the same argument to the inverse equivalence gives the reverse inclusion, and hence equality.
\end{proof}

\subsection{Proof of \Cref{thm:main}: reduction to the twisted K3 case}

We begin by recording the following standard consequence of lattice theory and derived Torelli.

\begin{lemma}\label{lem:untwist-k3}
Let $(S,\gamma)$ be a twisted K3 surface.  If $S$ has Picard number $\rho(S)\ge 13$, then there exists an untwisted K3 surface $S'$ such that $\Db(S,\gamma)\simeq\Db(S')$.
\end{lemma}

\begin{proof}
Consider the twisted transcendental lattice
\[
\rT(S,\gamma) \coloneqq \widetilde{\rH}_{\mathrm{alg}}(S,\gamma,\ZZ)^\perp \subseteq \widetilde{\rH}(S,\gamma,\ZZ).
\]
By Nikulin's lattice embedding theorem \cite[Theorem 1.12.2]{Nikulin}, there exists a primitive embedding
\[
\rT(S,\gamma) \hookrightarrow \Lambda_{\mathrm{K3}}.
\]
where $\Lambda_{\mathrm K3}$ is the unique even unimodular lattice of signature $(3,19)$.
The global Torelli theorem for K3 surfaces provides a K3 surface $S'$ such that there is a Hodge isometry $\rT(S,\gamma) \cong \rT(S')$, where $\rT(S')$ denotes the transcendental lattice of $S'$. This isometry extends to a Hodge isometry of Mukai lattices $\widetilde{\rH}(S,\gamma,\ZZ) \cong \widetilde{\rH}(S',\ZZ)$.  By the derived Torelli theorem for twisted K3 surfaces \cite[Corollary 3.21]{HuybrechtsMacriStellari}, such a Hodge isometry lifts to a derived equivalence $\Db(S,\gamma) \simeq \Db(S')$.
\end{proof}

\begin{remark} \label{rmk:induce-minus-brauer}
The equivalence $\Db(S,\gamma)\simeq\Db(S')$ induces a natural equivalence $\Db(S,-\gamma)\simeq\Db(S')$, given by the composition
\[
\Db(S') \xrightarrow{\ \mathrm{dual}\ } \Db(S')^{\mathrm{op}} \xrightarrow{\ \mathrm{op}\ } \Db(S,\gamma)^{\mathrm{op}} \xrightarrow{\ \mathrm{dual}\ } \Db(S,-\gamma).
\]
Hence an anti-autoequivalence of $\Db(S,\gamma)$ can be identified with an autoequivalence of the untwisted K3 surface $S'$.
\end{remark}

Now we proceed to the proof of \Cref{thm:main}.

\begin{proof}[Proof of \Cref{thm:main}]
We first suppose that $\Psi$ is an autoequivalence. By \Cref{prop:commute-involution} and \Cref{prop:albanese-kernel-action}, we may assume that $\Psi$ commutes with $\bL\iota^*$; this does not affect the statement we need to prove. Hence $\Psi$ descends, via \Cref{prop:descend-equivalence}, to an autoequivalence $\overline{\Psi}$ of $\Db\bigl(\km(X),\bar{\alpha}\bigr)$.

According to \Cref{prop:preserve-albanese-kernel},  there is a commutative diagram \cref{eq:Chow-diagram} relating the $\CH_0$-actions on $\km(X)$ and $X$, whose vertical isomorphisms are induced by $\Theta_{X,\alpha}^{\CH_0}$. Via the cohomological action $\Theta_{X,\alpha}^{\mathrm{tr}}$, the descent of $\Psi$ also yields a compatible commutative diagram on transcendental lattices:
\[
\begin{tikzcd}[sep=large]
\rT(\km(X),\bar{\alpha}) \ar[r,"\overline{\Psi}^{\mathrm{tr}}"] \ar[d,"\Theta_{X,\alpha}^{\mathrm{tr}}"'] &
\rT(\km(X),\bar{\alpha}) \ar[d,"\Theta_{X,\alpha}^{\mathrm{tr}}"] \\
\rT(X,\alpha) \ar[r,"\Psi^{\mathrm{tr}}"] &
\rT(X,\alpha).
\end{tikzcd}
\]

Since $\rank\bigl(\NS(\km(X))\bigr) \ge 17$, \Cref{lem:untwist-k3} implies that $(\km(X),\bar{\alpha})$ is derived equivalent to an untwisted K3 surface. On this untwisted K3 surface the statement of \Cref{conj:autoeq} is known to hold whenever the Picard number is at least $3$ (\cf \cite[Theorem 1.3]{LiYuZhang}). Hence, transferring the equality from the untwisted K3 surface to $(\km(X),\bar{\alpha})$, we obtain
\[
\overline{\Psi}^{\CH_0} = \pm\operatorname{id} \quad\Longleftrightarrow\quad \overline{\Psi}^{\mathrm{tr}} = \pm\operatorname{id}.
\]
The vertical isomorphisms further transfer this equivalence to $\Psi$, yielding
\[
\Psi^{\CH_0} = \pm\operatorname{id} \quad\Longleftrightarrow\quad \Psi^{\mathrm{tr}} = \pm\operatorname{id}.
\]
This proves the assertion for autoequivalences, as $\rT(X,\alpha)_{\QQ}$ is the smallest rational Hodge substructure containing the twisted $(2,0)$-line (see \Cref{subsection:FM-action}).

Now suppose $\Psi$ is an anti-autoequivalence. Consider the composition
\[
\DD \circ \Psi \colon \Db(X,\alpha) \longrightarrow \Db(X,\alpha)^{\mathrm{op}} \longrightarrow \Db(X,-\alpha).
\]
It suffices to prove that this equivalence acts as $\pm\id$ on $\CH_0(X)_{\alb}$. 
As before, we may assume that $\DD\circ\Psi$ commutes with $\bL\iota^*$. 
Hence it descends, via \Cref{prop:descend-equivalence}, to an equivalence
\[
\Db(\km(X),\bar{\alpha}) \longrightarrow \Db(\km(X),-\bar{\alpha}).
\]
By \Cref{rmk:induce-minus-brauer} together with the derived equivalence from \Cref{lem:untwist-k3}, this descended equivalence can be identified with an autoequivalence of an untwisted K3 surface $S'$. 
The statement then follows from the arguments for autoequivalences given above.
\end{proof}

\section{Bridgeland moduli spaces on (twisted) abelian surfaces}

In this section, we connect the surface-level statement proved in \Cref{thm:main} with birational automorphisms of $\mathrm{Kum}_n$-type varieties. We first recall the realization of such varieties as Albanese fibers of Bridgeland moduli spaces on twisted abelian surfaces. Then we combine \Cref{thm:torelli-C} with monodromy and wall-crossing to show that the relevant birational automorphisms are induced by derived (anti-)autoequivalences of the underlying twisted abelian surface. These ingredients altogether give \Cref{thm:main2}.

\subsection{Albanese fibers of Bridgeland moduli spaces}

We denote by $\Stab(X,\alpha)$ the space of Bridgeland stability conditions on $\Db(X,\alpha)$, which is connected \cite[Theorem 3.15]{HuybrechtsMacriStellari}.  
Given $\sigma \in \Stab(X,\alpha)$ and a Mukai vector $\mathbf{v} \in \widetilde{\rH}_{\mathrm{alg}}(X,\alpha,\mathbb{Z})$, we denote by $M_{\sigma}(X,\alpha,\mathbf{v})$ the moduli space of $\sigma$-semistable objects in $\Db(X,\alpha)$ with twisted Mukai vector $\mathbf{v}$ (see \cite{YoshiokaModuli, HuybrechtsMacriStellari, MYY2}).

Assume that $\mathbf{v} \in \widetilde{\rH}_{\mathrm{alg}}(X,\alpha,\mathbb{Z})$ is a primitive Mukai vector and $\sigma \in \Stab(X,\alpha)$ is $\mathbf{v}$-generic. There is an Albanese morphism
\begin{equation}
\mathfrak{alb} \colon M_{\sigma}(X,\alpha,\mathbf{v}) \longrightarrow \Alb\bigl(M_{\sigma}(X,\alpha,\mathbf{v})\bigr).
\end{equation}
We denote by $K_{\sigma}(X,\alpha,\mathbf{v})$ the fiber of $\mathfrak{alb}$ at $0$.

\begin{remark}
In the untwisted case, the Albanese variety is isomorphic to $X \times \widehat{X}$, and the Albanese map is given by (\cf \cite[Theorem 0.1]{YoshiokaModuli})
\[
[E] \longmapsto \bigl( \mathrm{codet}(E-E_0), \det(E-E_0) \bigr)
\]
where $[E_0] \in M_{\sigma}(X,\alpha,\mathbf{v})$ is a reference object. In the twisted case, see the forthcoming \cite{JLL} for the construction of the Albanese map and further details.
\end{remark}

\begin{theorem}[{\cite[Theorem 1.9]{YoshiokaPos}; see also \cite[Remark 4.5]{DebarreMacri}}]
\leavevmode \label{thm:mukai-morphism}
\begin{enumerate}
    \item When $\langle \mathbf{v},\mathbf{v} \rangle \ge 6$, $K_{\sigma}(X,\alpha,\mathbf{v})$ is nonempty of dimension $\langle \mathbf{v},\mathbf{v} \rangle -2$.  In this case, $K_{\sigma}(X,\alpha,\mathbf{v})$ is a smooth projective hyperk\"ahler variety, and the Mukai morphism gives an isomorphism
    \begin{equation}
    \theta_{\mathbf{v}} \colon \mathbf{v}^{\perp} \subset \widetilde{\rH}(X,\alpha,\mathbb{Z}) \longrightarrow \rH^2(K_{\sigma}(X,\alpha,\mathbf{v}),\mathbb{Z}).
    \end{equation}
    \item When $\langle \mathbf{v},\mathbf{v} \rangle = 4$, $K_{\sigma}(X,\alpha,\mathbf{v})$ is isomorphic to a K3 surface.  In this case, the Mukai morphism is not an isomorphism, and $\mathbf{v}^{\perp}$ is identified with the part of the cohomology of $K_{\sigma}(X,\alpha,\mathbf{v})$ coming from the abelian surface plus an extra class.
\end{enumerate}
\end{theorem}

For the rest of this section, unless otherwise stated, $\bfv$ is primitive, $\sigma$ is $\bfv$-generic, and $\langle \mathbf v,\mathbf v\rangle\ge 6$.

\subsection{The Markman--Mukai lattice}

The Markman--Mukai lattice was originally introduced in \cite{MarkmanMono}. Similar to the $\mathrm{K3}^{[n]}$ case, we adapt the following definition.

\begin{definition} \label{def:markman-mukai}
For a hyperk\"ahler variety $Y$ of $\mathrm{Kum}_n$-type, the \emph{Markman--Mukai lattice} is an even integral lattice $\widetilde{\Lambda}_{\mathrm{ab}} \simeq \UU(1)^{\oplus 4}$, together with a primitive Hodge embedding
\[
\rH^2(Y,\ZZ) \hookrightarrow \widetilde{\Lambda}_{\mathrm{ab}}.
\]
The algebraic Markman--Mukai lattice is the $(1,1)$-part $\widetilde{\Lambda}_{\mathrm{ab}}^{1,1}$.
\end{definition}

By the lattice-theoretic results of Nikulin \cite{Nikulin}, and Markman's description of the monodromy orbit \cite[Corollary 9.5]{MarkmanMono}, this embedding is unique up to the natural isometry action.
The orthogonal complement of $\rH^2(Y,\ZZ)$ is generated by a primitive vector $\bfv$ of square $2n+2$. When $Y$ is twisted modular, i.e. $Y = K_{\sigma}(X,\alpha,\bfv)$ for some twisted abelian surface $(X,\alpha)$, this lattice identifies with the usual Mukai lattice $\widetilde{\rH}(X,\alpha,\ZZ)$, and $\rH^2(Y,\ZZ)$ identifies with the orthogonal complement $\bfv^{\perp} \subset \widetilde{\rH}(X,\alpha,\ZZ)$.

Conjecturally, every $\mathrm{Kum}_n$-type hyperk\"ahler variety can be realized as the Albanese fiber of some moduli space of stable objects on a non-commutative abelian surface.
The Markman--Mukai lattice then controls such a category and its associated hyperkähler geometry. To be precise, derived equivalences, Fourier--Mukai transforms, wall-crossing, and moduli spaces are reflected as Hodge isometries and monodromy transformations of $\widetilde{\Lambda}_{\mathrm{ab}}$.

\subsection{Monodromy and numerically trivial automorphisms of $\mathrm{Kum}_n$-type varieties} \label{subsection:monodromy-and-numerical-trivial}

Let $Y$ be a hyperk\"ahler manifold.  The second cohomology $\rH^2(Y,\mathbb Z)$ carries the Beauville--Bogomolov--Fujiki pairing $q$, of signature $(3, b_2-3)$.

The \emph{monodromy group} $\Mon(Y)\subseteq \GL(\rH^*(Y,\mathbb Z))$ is generated by monodromy operators.  Denote by $\Mon^2(Y)$ the image of $\Mon(Y)$ in $\mathrm O(\rH^2(Y,\mathbb Z))$. For a hyperk\"ahler manifold of \emph{$\mathrm{Kum}_n$-type}, Markman proved the following description.

\begin{theorem}[{\cite[Theorem~1.4]{MarkmanMono}}]
\label{thm:monodromy-generalized-kummer}
For $Y$ of $\mathrm{Kum}_n$-type with $\dim_{\mathbb C}Y\ge4$, the image $\Mon^2(Y)$ equals
\[
\mathcal W^{\det\cdot\chi}\;\subset\;\mathrm O(\rH^2(Y,\mathbb Z)),
\]
where $\mathcal W$ is the group generated by reflections in vectors of square $\pm2$, and $\det\cdot\chi$ is the character sending each reflection to $-1$.
\end{theorem}

Now let $f\in \Bir(Y)$ be a birational automorphism. Its action on the second cohomology preserves the Beauville--Bogomolov--Fujiki form.
By standard deformation theory (see \cite[Corollary~9.43]{HuybrechtsBook}), the induced isomorphism $f^*:\rH^2(Y,\mathbb Z)\to \rH^2(Y,\mathbb Z)$ is a Hodge isometry and can be realized as a parallel transport operator; hence $f^*$ belongs to the monodromy group $\Mon^2(Y)$.

As a consequence, we obtain the following result.

\begin{proposition}\label{prop:bir-admissible}
Let $(X,\alpha)$ be a twisted abelian surface and $Y=K_\sigma(X,\alpha,\mathbf{v})$ the Albanese fiber of a Bridgeland moduli space on $(X,\alpha)$ with $\langle\bfv,\bfv\rangle\geq 6$.  
For any $f\in \Bir(Y)$, the induced isometry $f^*\in \mathrm O(\rH^2(Y,\mathbb Z))$ extends to an admissible Hodge isometry in $\mathrm{SO}\bigl(\widetilde{\rH}(X,\alpha,\mathbb Z)\bigr)$, where $\rH^2(Y,\mathbb Z)$ is naturally identified with a sublattice of the twisted Mukai lattice $\widetilde{\rH}(X,\alpha,\mathbb Z)$.
\end{proposition}

\begin{proof}
By \Cref{thm:monodromy-generalized-kummer}, $f^*$ is a product of an even number of reflections in vectors of square $\pm2$ (up to $-\id$).  
Each such reflection  on $\rH^2(Y,\mathbb Z)$ extends to a reflection (as well as $-\id$) on the whole $\widetilde{\rH}(X,\alpha,\mathbb Z)$. 
Hence the product of an even number of such extended reflections has determinant $1$ and therefore lies in $\mathrm{SO}\bigl(\widetilde{\rH}(X,\alpha,\mathbb Z)\bigr)$. Since $f^*$ is orientation-preserving, the extended product also preserves the chosen positive orientation of $\widetilde{\rH}(X,\alpha,\RR)$.  Hence the extended Hodge isometry is admissible.
\end{proof}

Next, we characterize birational automorphisms whose action on the second cohomology is trivial. For any hyperk\"ahler variety $Y$ of $\mathrm{Kum}_n$-type, define
\[
\Aut_0(Y) \coloneqq \ker\bigl(\Aut(Y) \to \mathrm{O}(\rH^2(Y,\mathbb{Z}))\bigr)
\]
to be the group of \emph{numerically trivial automorphisms}. It has been shown by Hassett-Tschinkel that $\Aut_0(Y)$ is deformation invariant \cite[Theorem 2.1]{HT}.  Moreover,  Boissière, Nieper-Wißkirchen and Sarti \cite[Corollary 5]{BNS11} have shown that
\[
\Aut_0(Y) \cong (\mathbb{Z}/(n+1)\mathbb{Z})^4 \rtimes \mathbb{Z}/2\mathbb{Z}.
\]

When $Y = K_{\sigma}(X,\alpha,\mathbf{v})$, the group $\Aut_0(Y)$ admits an explicit geometric description. It is generated by:
\begin{itemize}
    \item Translations $t_a$ for $a \in X[n+1]$, acting on sheaves by $E \mapsto t_a^*E \otimes \mathcal{L}$, where $\mathcal{L}$ is a suitable line bundle that adjusts the determinant to keep the Mukai vector fixed.
    \item The involution $(-1)^*$ induced by the inversion map $x \mapsto -x$ on $X$, acting on sheaves by pullback.
\end{itemize}
Both classes of generators are naturally realized by derived autoequivalences.

The presence of numerically trivial automorphisms is one of the main differences from the $\mathrm{K3}^{[n]}$ case, and accounts for the additional arguments required in the proofs.

\subsection{Wall-crossing for the Albanese fibers $K_\sigma(X,\alpha,\bfv)$}

We next recall the wall-crossing result needed to pass between Albanese fibers for different generic stability conditions. This follows from the wall-crossing theory developed in \cite{BayerMacri, MYY2}.

\begin{theorem}[Minamide-Yanagida-Yoshioka] \label{thm:wall-crossing-to-equivalence}
Let $(X,\alpha)$ be a twisted abelian surface.  Let $\sigma, \tau \in \Stab(X,\alpha)$ be generic stability conditions with respect to a primitive Mukai vector $\mathbf{v}$ with $\langle\bfv,\bfv\rangle\geq 6$. There is a birational map $K_{\sigma}(X,\alpha,\mathbf{v}) \dashrightarrow K_{\tau}(X,\alpha,\mathbf{v})$ that is induced by a derived (anti-)autoequivalence $\Psi$, i.e. $E\mapsto \Psi(E)$ such that the action $\Psi^{\widetilde{\rH}}$ is identity or $-\DD$.
\end{theorem}

\begin{proof}
The proof essentially follows from the wall-crossing analysis in \cite{MYY2}. 
It is shown in \cite[Theorem 5.4.1]{MYY2} that there exists a birational map
\[
f : M_{\sigma}(X,\alpha,\mathbf{v}) \dashrightarrow M_{\tau}(X,\alpha,\mathbf{v})
\]
between moduli spaces, together with a derived (anti-)autoequivalence 
$\Psi : \Db(X,\alpha) \to \Db(X,\alpha)$ such that
\[
f(E) = \Psi(E)
\]
for a general point $E \in K_{\sigma}(X,\alpha,\mathbf{v})$.

The autoequivalence $\Psi$ is a composition of elementary equivalences described in \cite[Proposition 5.1.5]{MYY2}. 
Specifically, when crossing a single wall $W_{w_1}$ from $\sigma$ to an adjacent chamber, where $w_1$ is an isotropic Mukai vector satisfying $\langle w_1,w_1\rangle = 0$ and $\langle \mathbf{v}, w_1\rangle = 1$ (note that in this case $\alpha = 0$), the corresponding elementary contravariant equivalence is given by
\[
\Db(X) \xrightarrow{\Psi_{\mathcal{E}}} \Db(X') \xrightarrow{\DD \circ \Psi_{\mathcal{E}^\vee}} \Db(X),
\]
where $X' = M_{(X,\alpha)}(w_1)$ is the fine moduli space of $\alpha$-twisted stable objects, $\mathcal{E}$ is the universal twisted sheaf on $X \times X'$, and $\Psi_{\mathcal{E}}$ denotes the Fourier--Mukai transform. 
Such an equivalence is an involution and acts as $-\DD$ on the twisted Mukai lattice. 
Since any two generic stability conditions $\sigma$ and $\tau$ can be connected by a path crossing only finitely many such walls, the overall equivalence $\Phi$ is a finite composition of these elementary equivalences.

Finally, the assertion that $\Psi$ induces a birational map between the Albanese kernels follows from \Cref{lem:adjust-albanese-fiber}. 
\end{proof}

\begin{lemma} \label{lem:adjust-albanese-fiber}
Let $\Psi$ be an (anti-)autoequivalence of $\Db(X,\alpha)$. There exists $\gamma \in A_{X,\alpha}$ such that the birational map induced by $\Phi_{\gamma} \circ \Psi$ preserves the Albanese fiber. Moreover, the cohomological action of $\Phi_{\gamma}\circ \Psi$ on the Mukai lattice is the same as that of $\Psi$.
\end{lemma}

\begin{proof}
Let $f \colon M_{\sigma}(X,\alpha,\bfv) \dashrightarrow M_{\tau}(X,\alpha,\bfv)$ be the birational map induced by $\Psi$. For simplicity, set $M_{\sigma} = M_{\sigma}(X,\alpha,\bfv)$ and $M_{\tau} = M_{\tau}(X,\alpha,\bfv)$. 
By functoriality of the Albanese morphism, there exists an isomorphism $\widetilde{f} \colon \Alb(M_{\sigma}) \to \Alb(M_{\tau})$ and a point $c \in \Alb(M_{\tau})$ such that
\[
\mathfrak{alb}_{\tau}(f(E)) = \widetilde{f} \bigl( \mathfrak{alb}_{\sigma}(E) \bigr) + c,
\]
on the domain of definition of $f$. In particular, $f$ maps the fiber $\mathfrak{alb}_{\sigma}^{-1}(0)$ birationally onto the fiber $\mathfrak{alb}_{\tau}^{-1}(c)$.

According to \cite[\S 10.1]{MarkmanKum}, we have a morphism defined by 
\[
\begin{aligned}
    T \colon X\times\widehat{X} &\longrightarrow \Alb(M_{\tau}) \\
    (u,\nu) &\longmapsto \mathfrak{alb}_{\tau}(\Phi_{(u,\nu)}\circ \Psi(E)) - \mathfrak{alb}_{\tau}(\Psi(E))
\end{aligned}
\]
where $E \in M_{\sigma}$ is a general point. This morphism is independent of the choice of $E$. 

\setcounter{claimnum}{0}
\claim
$\Phi_{\gamma}$ fixes every stability condition in $\Stab(X,\alpha)$, for every $\gamma=[(u,\nu)]\in A_{X,\alpha}$.

By \Cref{prop:twisted-MP}, $\Phi_{\gamma}$ acts trivially on the Mukai lattice for every $\gamma=[(u,\nu)]\in A_{X,\alpha}$; in particular, it fixes every Mukai vector. Hence it preserves the central charge of any stability condition. Since a stability condition on a twisted abelian surface is determined by its central charge (\cf \cite[Remark 1.5.8]{MYY2}), it follows that $\Phi_{\gamma}$ acts trivially on $\Stab(X,\alpha)$.

\claim
$T$ is surjective.

Since both $X \times \widehat{X}$ and $\Alb(M_{\tau})$ are abelian fourfolds, it is enough to show that $T$ is non-degenerate. 
Recall that there is a morphism
\[
\Sigma \colon M_{\tau} \longrightarrow X \times \widehat{X}, \quad F \longmapsto \bigl( \mathfrak{alb}_{\tau}(Nc_2(F)-Nc_2(F_0)), \det(F) \otimes \det(F_0)^{-1} \bigr),
\]
where $F_0 \coloneqq \Psi(E)$ is the reference point, and $N$ is a sufficiently large integer such that $Nc_2(F)$ is an integral class for every $F \in M_{\tau}$. Such a common $N$ exists as the possible denominator comes from the fixed twisting data of the twisted Chern character. By the universal property of Albanese variety, $\Sigma$ factors as $\Sigma' \circ \mathfrak{alb}_{\tau}$. On the other hand, we have an orbit morphism 
\[
o \colon X \times \widehat{X} \longrightarrow M_{\tau}, \quad (u,\nu) \longmapsto \Phi_{(u,\nu)}(\Psi(E)).
\]
Notice that $\Sigma \circ o = \Sigma' \circ NT$ by our construction. Hence it suffices to prove that $\Sigma \circ o$ is non-degenerate.
For any twisted sheaf $F \in M_{\tau}$, its tangent space is given by $T_F M_{\tau} = \Ext^1(F,F) = \rH^1(X,\bR\mathcal{H}om(F,F))$. Moreover, the endomorphism complex $\bR\mathcal{H}om(F,F)$ has trivial twisting, and descends to an ordinary complex on $X$. Therefore, the tangent map $d(\Sigma \circ o)_0$ is calculated the same way as in the untwisted case, which has rank $4$ by \cite[Lemma 4.3]{YoshiokaModuli}. This proves the non-degeneracy.
\smallskip

Choose $(u_0,\nu_0) \in X \times \widehat{X}$ such that $T(u_0,\nu_0) = -c$. Let $g \colon M_{\sigma} \dashrightarrow M_{\tau}$ be the birational map induced by $\Phi_{(u_0,\nu_0)} \circ \Psi$. Then $\widetilde{g} = \widetilde{f}$, because $\Phi_{(u_0,\nu_0)}$ acts on the Albanese variety by translation. We have
\[
\begin{aligned}
\mathfrak{alb}_{\tau}(g(E)) &= \mathfrak{alb}_{\tau}(\Phi_{(u_0,\nu_0)}\circ \Psi(E)) - \mathfrak{alb}_{\tau}(\Psi(E)) + \mathfrak{alb}_{\tau}(\Psi(E))\\
&= -c + \widetilde{f}(\mathfrak{alb}_{\sigma}(E)) + c \\
&= \widetilde{g}(\mathfrak{alb}_{\sigma}(E)).
\end{aligned}
\]
This completes the proof.
\end{proof}

\subsection{Lifting birational automorphisms to derived autoequivalences}

This subsection shows that every birational automorphism on a twisted modular $\mathrm{Kum}_n$-type variety is induced by a derived (anti-)autoequivalence. This is the main bridge from \Cref{thm:main} to \Cref{thm:main2}.

\begin{theorem} \label{thm:birational-to-equivalence}
Let $Y=K_\sigma(X,\alpha,\bfv)$. Every birational automorphism $f \in \Bir(Y)$ admits a lifting to an (anti-)autoequivalence
\[
\Psi \colon \Db(X,\alpha) \longrightarrow \Db(X,\alpha) \quad \text{or} \quad \Psi\colon \Db(X,\alpha) \longrightarrow \Db(X,\alpha)^{\rm op}
\]
in the sense that there exists an open subset $U \subset Y$ such that $f(E) = \Psi(E)$ for any $E \in U$.
\end{theorem}

\begin{proof}
The proof is analogous to the case of $\mathrm{K3}^{[n]}$-type hyperk\"ahler varieties, but needs additional attention due to the existence of numerically trivial automorphisms.

The birational automorphism $f \in \Bir(Y)$ induces an isometry $f^* : \rH^2(Y,\mathbb Z) \to \rH^2(Y,\mathbb Z)$.  
By \Cref{prop:bir-admissible}, this isometry extends to an admissible Hodge isometry
\begin{equation}\label{eq:hodge}
    \widetilde{f}:\widetilde{\rH}(X,\alpha,\ZZ)\to \widetilde{\rH}(X,\alpha,\ZZ),
\end{equation}
It sends the vector $\bfv$ to $\pm \bfv$. According to \Cref{thm:torelli-C}, if $\tilde{f}(\bfv)=\bfv$, then $\widetilde f$ lifts to an autoequivalence $\Psi:\Db(X,\alpha)\rightarrow \Db(X,\alpha)$. 
If $\tilde{f}(\bfv)=-\bfv$, then $\widetilde f$ lifts to an anti-autoequivalence $\Psi:\Db(X,\alpha)\rightarrow \Db(X,\alpha)^{\rm op}$.  

The equivalence $\Psi$ (resp. $\Psi[1]$) induces an isomorphism $\psi: Y\rightarrow Y':=K_{\Psi(\sigma)}(X,\alpha,\bfv)$. Due to \Cref{thm:wall-crossing-to-equivalence}, there exists a derived (anti-)autoequivalence $\Psi'$ and a birational map
\[
Y'\dashrightarrow Y, \quad E \longmapsto \Psi'(E)
\]
defined over an open subset $U\subset Y'$, such that $\Psi'$ acts trivially on $\rH^2(Y,\ZZ)$.

Composing it with $Y\xrightarrow{\psi} Y'$, we obtain a birational automorphism
$$\varphi:Y \xlongrightarrow{\sim} Y'\dashrightarrow Y,$$ which sends $E$ to $\Phi(E)$ for $E$ lying in some open subset of $M$, where $\Phi=\Psi'\circ \Psi$ (resp. $\Phi=\Psi' \circ \Psi[1]$).
The actions of $f^*$ and $\varphi^*$ on $\rH^2(Y,\ZZ)$ differ by products of prime exceptional reflections by \cite[Theorem 6.18(5)]{MarkmanMono}. Since both maps land in the same chamber, $f^*$ and $\varphi^*$ coincide, and $\varphi$ and $f$ must differ by an element in $\Aut_0(Y)$. By the explicit description in \Cref{subsection:monodromy-and-numerical-trivial} and Claim 1 in the proof of \Cref{lem:adjust-albanese-fiber}, every element of $\Aut_0(Y)$ lifts to an autoequivalence of $\Db(X,\alpha)$ preserving both the stability condition and the Mukai vector. The assertion follows.
\end{proof}

We will also say that $f$ is induced by the derived (anti-)autoequivalence $\Psi$ in this setting.

\begin{remark}\label{rmk:symplectic-type}
Let $\varepsilon$ be defined by $\Psi^{\widetilde{\rH}}(\bfv)= \tilde{f}(\bfv) = \varepsilon\bfv$. As the proof of \Cref{thm:birational-to-equivalence} shows, the symplectic type of $\Psi$ is determined by that of $f$ and by $\varepsilon$:
\[
\begin{array}{c|cc}
 & f\ \mathrm{symplectic} & f\ \mathrm{anti\text{-}symplectic} \\ \hline
\varepsilon=1
& \mathrm{symplectic\ autoequivalence}
& \mathrm{anti\text{-}symplectic\ autoequivalence} \\
\varepsilon=-1
& \mathrm{anti\text{-}symplectic\ anti\text{-}autoequivalence}
& \mathrm{symplectic\ anti\text{-}autoequivalence}
\end{array}
\]
\end{remark}

\subsection{Proof of \Cref{thm:main2}: from derived actions to Bloch's conjecture}

The assertion follows from the following ingredients.

\begin{description}[font=\bfseries]

\item[$\bullet$ Lifting birational automorphisms to autoequivalences]
    According to \Cref{thm:birational-to-equivalence}, 
    $f$ can be lifted to an (anti-)autoequivalence
    \[
    \Psi \in \Aut\bigl(\Db(X,\alpha)\bigr) \quad\text{or}\quad \Psi \colon \Db(X,\alpha) \longrightarrow \Db(X,\alpha)^{\mathrm op}.
    \]

\item[$\bullet$ Action on $0$-cycles]
    For any closed point $E \in Y$ (viewed as a complex on $(X,\alpha)$), we have
    \[
    f_*(E) = \Psi(E) \quad \text{in } \CH_0(K).
    \]
    Recall that the action of $\Psi^{\CH}$ on $\CH_0(X)_{\mathrm{alb}}$ is given by
    \[
    \Psi^{\CH}\bigl(c_2(E)-c_2(E_0)\bigr) = c_2\bigl(\Psi(E)\bigr)-c_2\bigl(\Psi(E_0)\bigr).
    \]

\item[$\bullet$ A criterion of Marian-Zhao]
    A consequence of the main theorem in \cite{MarianZhao} is the following.
\end{description}

    \begin{theorem}\label{thm:Marian-Zhao}
    Let $Y = K_{\sigma}(X,\alpha,\mathbf{v})$ be a Bridgeland moduli space on $(X,\alpha)$.  If $E, F$ are two points in $Y$, then
    \[
    c_2(E) = c_2(F) \;\text{ in }\; \CH_0(X) \qquad\Longleftrightarrow\qquad
    \chow{E} = \chow{F} \;\text{ in }\; \CH_0(Y).
    \]
    \end{theorem}
    \begin{proof}
    Set $M = M_{\sigma}(X,\alpha,\mathbf{v})$.  By the main result of \cite{MarianZhao}, we have the following equivalence
    \[
    c_2(E) = c_2(F) \text{ in } \CH_0(X) \quad\Longleftrightarrow\quad
    \chow{E} = \chow{F} \text{ in } \CH_0(M).
    \]
    There is an inclusion map $i \colon Y \hookrightarrow M$ and a diagram
    \[
    \begin{tikzcd}
    \Alb(M) \times Y \ar[r,"\pi"] \ar[d] & M \ar[d,"\operatorname{alb}"] \\
    \Alb(M) \ar[r,"\times (n+1)"] & \Alb(M)
    \end{tikzcd}
    \]
    For a point $E \in Y$ (viewed as a complex on $(X,\alpha)$), its image in $\CH_0(M)$ is denoted by $\chow{E}$.

    Using the pushforward of the inclusion map $i_* : \CH_0(Y) \hookrightarrow \CH_0(M)$, the forward direction is clear: $\chow{E} = \chow{F}$ in $\CH_0(Y)$ implies $c_2(E)=c_2(F)$ in $\CH_0(X)$.

    Conversely, suppose $c_2(E)=c_2(F)$ in $\CH_0(X)$.  Then we have $\chow{E} = \chow{F}$ in $\CH_0(M)$.  Consider the composition
    \[
    \CH_0(M) \xrightarrow{\pi^*} \CH_0(\Alb(M) \times Y) \xrightarrow{q_*} \CH_0(Y),
    \]
    where $q \colon \Alb(M) \times Y \to Y$ is the projection onto the second factor.  It sends each point $\chow{E} \in \CH_0(Y)$ to a multiple of $\chow{E}$.  It follows  that $\chow{E} = \chow{F}$ in $\CH_0(Y)$.
    \end{proof}

Now the proof of \Cref{thm:main2} is formal.

\begin{proof}[Proof of \Cref{thm:main2}]
Let $f\in\Bir(Y)$ be symplectic. By \Cref{thm:birational-to-equivalence}, $f$ is induced by an (anti-)autoequivalence $\Psi$ of $\Db(X,\alpha)$. Since $f$ is symplectic, $\Psi$ is a symplectic autoequivalence or an anti-symplectic anti-autoequivalence by \Cref{rmk:symplectic-type}. Hence \Cref{thm:main} implies that $\Psi$ acts trivially on $\CH_0(X)_{\mathrm{alb}}$. It follows from \Cref{thm:Marian-Zhao} that $f_*$ acts as the identity on $\CH_0(Y)$, which is the assertion of \Cref{thm:main2}.
\end{proof}

\section{Shen--Yin--Zhao filtrations and Bloch's conjecture for anti-symplectic maps}

In this section, we first construct a Shen--Yin--Zhao filtration on twisted modular $\mathrm{Kum}_n$-type varieties. We then explain why ``SYZ = Voisin'' (\Cref{conj:SYZ}) would imply the expected Bloch-type action for anti-symplectic maps (\Cref{conj:anti-symplectic}). Although this equality remains conjectural in general, we prove \Cref{conj:SYZ} for $K_3(X)$. Moreover, we obtain an unconditional result for \Cref{conj:anti-symplectic} over twisted modular sixfolds using Floccari's construction \cite{Fl24}.

\subsection{The Shen--Yin--Zhao type filtration from O'Grady's filtration}

We begin with the surface-level filtration from which our construction is inherited. Let $S$ be a K3 surface. O'Grady introduced in \cite{OG13} a set-theoretic filtration $\mathbf S_\bullet(S)$ on $\CH_0(S)$ by setting
\[
\mathbf{S}_i(S) = \bigcup_{\substack{\text{effective} \\ \deg(z)=i}} \Big\{ z + \mathbb{Q} \cdot \fro_S \Big\},
\]
where $\fro_S$ is the Beauville--Voisin class. 

For an abelian surface $X$, an O'Grady-type set-theoretic filtration is defined by  
\[
\mathbf{S}_i(X) := \mathbf{S}_i(\mathrm{Km}(X)) \subseteq \CH_0(X),
\]
where $\CH_0(\mathrm{Km}(X))$ is identified as a subset of $\CH_0(X)$ via the pullback of the rational map $\varpi: X \dashrightarrow \mathrm{Km}(X)$. The following simple observations hold:
\begin{itemize}
    \item $\mathbf{S}_0(X)$ consists of the multiples of $\{0\}$.
    \item Let $z \in \CH_0(X)_{(0)} \oplus \CH_0(X)_{(2)}$. Then $z \in \mathbf{S}_i(X)$ if and only if $\varpi_*(z) \in \mathbf{S}_i(\mathrm{Km}(X))$.
\end{itemize}
The verification is straightforward.

Now let $(X,\alpha)$ be a twisted abelian surface. Let
$(\km(X),\bar\alpha)$ be the associated twisted Kummer surface. Denote by
$\fro_{(\km(X),\bar\alpha)}$ the twisted Beauville--Voisin class introduced in
\cite[Definition 4.1]{CLZZ:2024}, and let
\[
 \fro_{(X,\alpha)}\in \CH_0(X)
\]
be its pullback under the Kummer correspondence. 

Denote by $c_2(E)_{(2)}$ the projection of $c_2(E)$ onto $\CH_0(X)_{(2)}$. One can see that 
\[
c_2(E)_{(2)} = \frac{1}{2}\bigl(c_2(E) + \iota^*c_2(E)\bigr) - k\chow{0},
\]
where $k = \deg c_2(E)$. The appearance of $c_2(E)_{(2)}$ is consistent with the Kummer-surface description, since cycles pulled back from the Kummer surface contribute only to the even Beauville components of $\CH_0(X)$, and in particular have no $(1)$-component.

\begin{definition}\label{def:SYZ-filtration}
We define a derived filtration on $\Db(X,\alpha)$:
\begin{equation}\label{eq:SYZ-filtration}
    \mathbf{S}_i(\Db(X,\alpha)) = \left\{ E \in \Db(X,\alpha) \;\middle|\; c_2(E)_{(2)} - \operatorname{rank}(E)\,\fro_{(X,\alpha)} \in \mathbf{S}_i(X) \right\}.
\end{equation}
Let $M = M_\sigma(X,\alpha,\mathbf{v})$ and let $K := K_\sigma(X,\alpha,\mathbf{v})$ denote the kernel of the Albanese map.
Then we define $\mathbf{S}^{\mathrm{SYZ}}_{\bullet} \CH_0(K)$ by restricting the filtration $\mathbf{S}_{\bullet}(\Db(X,\alpha))$ to $\CH_0(K)$:
\begin{equation}\label{eq:SYZ}
    \mathbf{S}^{\mathrm{SYZ}}_{i}\CH_0(K) = \left\langle E \in K \;\middle|\; E \in \mathbf{S}_i(\Db(X,\alpha)) \right\rangle.
\end{equation}
\end{definition}

It is natural to ask whether this filtration agrees with Voisin's filtration. This is the $\mathrm{Kum}_n$-type analogue of the equality between the Shen--Yin--Zhao and Voisin filtrations known in the $\mathrm{K3}^{[n]}$-type case. We propose the following conjecture.

\begin{conjecture}\label{conj:SYZ}
Set $\dim K = 2n$. For any point $E \in K$, we have
\[
\dim O_{E} \geq n - i \quad\Longleftrightarrow\quad c_2(E)_{(2)} - \operatorname{rank}(E)\,\fro_{(X,\alpha)} \in \mathbf{S}_i(X).
\]
Here $O_E$ denotes the rational equivalence orbit of the point $E \in K$, and $\dim O_E$ denotes the maximal dimension of a constant-cycle subvariety through $E$. In particular, $\mathbf{S}^{\rm SYZ}_\bullet\CH_0(K) = \mathbf{S}_\bullet\CH_0(K)$.
\end{conjecture}

\begin{remark}
Using the degeneration sequence for twisted sheaves as in \cite[Proposition 4.4, \S 5.5]{CLZZ:2024}, one can reduce the twisted case to the untwisted case.  
Thus \Cref{conj:SYZ} is expected to hold for all twisted abelian surfaces once the corresponding untwisted statement is known. We will give details in the upcoming paper \cite{JLL}.
\end{remark}

\subsection{Application of \Cref{conj:SYZ} towards Bloch's conjecture}

We first explain how the conjectural comparison would imply the desired Bloch-type statement. This step is essentially formal: once the SYZ filtration agrees with Voisin's filtration, the anti-symplectic sign rule can be checked on the Kummer model, where the filtration admits a standard splitting.

\begin{definition}
Recall that Voisin's filtration admits a standard splitting, with components given by (\cf \cite[Theorem 2.5]{VoisinHK})
\[
\CH_0(\km(X)^{[n]})_{i} = \left\langle (o,\ldots,o, \chow{z_1}-o,\ldots, \chow{z_i}-o) \mid z_1,\ldots, z_i \in X \right\rangle \subset \CH_0(\km(X)^{[n]}).
\]
Thus every zero-cycle $\xi \in \CH_0(\km(X)^{[n]})$ admits a decomposition
\[
 \xi=\sum_{i=0}^n \xi_i,
 \qquad \xi_i\in \CH_0\bigl(\km(X)^{[n]}\bigr)_i.
\]  
If $\xi$ is represented by the class $\chow{z}$ of a point $z\in \km(X)^{[n]}$, we simply write $\xi_i=\chow{z}_i$.
\end{definition}

Consequently, we obtain the following implication.

\begin{proposition} \label{prop:SYZ-give-anti}
Let $Y$ be a $\mathrm{Kum}_n$-type hyperk\"ahler variety. Suppose \Cref{conj:SYZ} holds. Then \Cref{conj:anti-symplectic} holds whenever $Y$ is twisted modular. 
\end{proposition}

\begin{proof}
Suppose every object in $Y$ has rank $r$. Define $c'(E) := c_2(E)_{(2)} - r\,\fro_{(X,\alpha)}$. By \Cref{conj:SYZ}, we have
\[
c'(E) \in \mathbf{S}_n(\mathrm{Km}(X)),
\]
via the identification $\CH_0(\mathrm{Km}(X)) = \CH_0(X)_{(0)} \oplus \CH_0(X)_{(2)}$.

Consider the incidence variety
\[
R = \Big\{\, (E,\xi) \mid c'(E) = \chow{\operatorname{supp}\xi} + k \cdot \mathfrak{o}_X \in \CH_0(\km(X)) \,\Big\} \subset Y \times \km(X)^{[n]}.
\]
By \Cref{conj:SYZ} and an argument analogous to that of \cite[Section 2.2]{SYZ20}, there exists an irreducible component $R_0$ of $R$ that dominates both factors. More precisely, the case $i=n$ of \Cref{conj:SYZ} implies that the projection $R \to \km(X)^{[n]}$ is dominant. On the other hand, the argument of \cite[Proposition 2.2]{OG13} shows that the projection $R \to Y$ is also dominant. Hence one may choose an irreducible component $R_0 \subset R$ that dominates both $\km(X)^{[n]}$ and $Y$.
The correspondence $R_0$ induces an isomorphism
\[
\eta: \CH_0(Y) \stackrel{\cong}{\longrightarrow} \CH_0(\km(X)^{[n]}),
\]
which, by \Cref{conj:SYZ}, yields a splitting of Voisin's filtration on $\CH_0(Y)$.

Now let $f \in \Bir(Y)$ be an anti-symplectic birational automorphism. Take a general point $y \in Y$. By the surjectivity of $R_0$, we can find points $x, x' \in \km(X)^{[n]}$ such that
\[
\chow{x} = \eta(\chow{y}), \qquad \chow{x'} = \eta(\chow{f(y)}) \quad \text{in } \CH_0(\km(X)^{[n]}).
\]
By the definition of $R$, the following diagram commutes:
\[
\begin{tikzcd}[row sep=large, column sep=large]
\CH_0(Y) \arrow[r, "\eta"] \arrow[d, "c'"] & \CH_0(\km(X)^{[n]}) \arrow[d, "c_2"] \\
\CH_0(X) \arrow[d] & \CH_0(\km(X)) \arrow[d] \\
\CH_0(X)_{\mathrm{alb}} \arrow[r] & \CH_0(\km(X))_{\mathrm{alb}}
\end{tikzcd}
\]
Since $f$ is anti-symplectic, its lift to a derived (anti-)equivalence acts as $-\operatorname{id}$ on the Albanese kernel by \Cref{rmk:symplectic-type} and \Cref{thm:main}. It follows that $\chow{x'}_1 = -\chow{x}_1$ in the splitting decomposition. Applying \cite[Theorem 4.3(iii)]{LiYuZhang} then yields
\[
\chow{f(y)}_i = (-1)^i \chow{y}_i \quad \text{in } \CH_0(Y)
\]
for all $i$. This is exactly the statement of \Cref{conj:anti-symplectic} for anti-symplectic birational automorphisms.
\end{proof}

\subsection{Filtrations on $\mathrm{Kum}_n(X)$}

We now turn from the general conjectural framework to an explicit test case. For generalized Kummer varieties, points admit a concrete description as unordered tuples on the abelian surface with sum zero. Thus \Cref{conj:SYZ} admits a very explicit geometric reinterpretation.

\begin{conjecture}
For points $z_m \in X$ with $0 \le m \le n$, set $z = \{z_0,\ldots,z_n\} \in K_n(X)$. Then $\dim O_z \ge n - i$ if and only if there exist points $a_1,\dots,a_i \in X$ and an integer $k$ such that
\[
\sum_{m=0}^n \chow{z_m} = \sum_{j=1}^i \bigl(\chow{a_j} + \chow{-a_j}\bigr) + k \chow{0} \in \CH_0(X).
\]
In particular, for any $z \in X_0^{n+1}$, we always have
\[
\sum_{m=0}^n \chow{z_m} = \sum_{j=1}^n \bigl(\chow{a_j} + \chow{-a_j}\bigr) + k\chow{0} \in \CH_0(X)
\]
for some $a_1,\dots,a_n \in X$ and $k \in \mathbb{Z}$.
\end{conjecture}

In the following, we confirm this conjecture when $n = 3$. 
For $K_3(X)$, there is a natural rational map to the Hilbert scheme of the associated Kummer surface. On the open locus of $K_3(X)$ parametrizing reduced subschemes
\[
z=\{a_0,a_1,a_2,a_3\},\qquad \sum_{i=0}^3 a_i=0,
\]
with the relevant points lying in the domain of the rational map $\varpi:X\dashrightarrow\km(X)$, define
\[
\rho\colon K_3(X)\dashrightarrow \operatorname{Km}(X)^{[3]}
\]
by
\begin{equation} \label{eq:kummer-sym3-map}
\rho(z)=\Big\{\varpi(a_0+a_1), \ \varpi(a_0+a_2), \ \varpi(a_0+a_3)\Big\}.
\end{equation}
Changing the distinguished point $a_0$ only changes the three points by signs and permutation, and hence gives the same unordered triple in $X/\pm1$. Thus $\rho$ is well defined on a dense open subset and extends to a rational map
\[
K_3(X)\dashrightarrow \operatorname{Km}(X)^{[3]}.
\]
This is the map used below to compare the rational-orbit filtration on $K_3(X)$ with O'Grady's filtration on the Kummer surface.
Now we are ready to derive the following result.

\begin{theorem}\label{thm:SYZ-sixfold}
If $K = K_3(X)$, then $\mathbf{S}_\bullet^{\rm SYZ} \CH_0(K) = \mathbf{S}_\bullet \CH_0(K)$.
\end{theorem}

\begin{proof}
\setcounter{proofstep}{0}

Let $z = \{a_0, a_1, a_2, a_3\} \in K_3(X)$ be a point lying in the domain of the rational map 
$\rho: K_3(X) \dashrightarrow \Km(X)^{[3]}$, where $a_i \in X$ and $\sum_{i=0}^3 a_i = 0$. 
Its Chern class is given by 
\[
c_2(z) = \sum_{i=0}^3 \chow{a_i} \in \CH_0(X).
\]
There is an alternating sum relation given by
\[
\sum_{I \subseteq \{0,\dots,3\}} (-1)^{|I|}\ \chow{ \sum_{i \in I} a_i } = 0 \quad \text{in } \CH_0(X),
\]
because the left-hand side lies in the 4th Pontryagin power of $\CH_0(X)_{\ge 1}$, which is zero \cite[Theorem 0.1]{BlochAbelian}. Expanding yields
\begin{equation}\label{eq:6-key}
\chow{0} - \sum_i \chow{a_i} + \sum_{i<j} \chow{a_i + a_j} - \sum_{i<j<k} \chow{a_i + a_j + a_k} + \chow{0} = 0.
\end{equation}
Observe that, by \cite[Lemma 2.1]{LinCorrigendum},
\[
\sum_{i<j<k} \bigl( \chow{a_i + a_j + a_k} - \chow{0} \bigr) = \sum_{i=0}^3 \bigl( \chow{-a_i} - \chow{0} \bigr) = \sum_{i=0}^3 \bigl( \chow{a_i} - \chow{0} \bigr).
\]
Substituting this into \eqref{eq:6-key}, we obtain the fundamental relation
\begin{equation}\label{eq:alt-sum}
c_2(z)=\sum_{i=0}^3 \chow{a_i} = \frac{1}{2} \sum_{i<j} \chow{a_i + a_j} + \chow{0} \quad \text{in } \CH_0(X).
\end{equation}

Next, we claim that $\dim O_z = \dim O_{\rho(z)}$. This follows easily from the facts below:
\begin{itemize}
    \item[(i)] Maximal dimensional irreducible components of $O_{\rho(z)}$ are Zariski dense in $\Km(X)^{[3]}$ (see \cite[Lemma 20]{Voisin15}).
    \item[(ii)] $\rho(O_z) \subseteq O_{\rho(z)}$ and $\rho^{-1}(O_{\rho(z)}) \subseteq O_z$.
\end{itemize}
To verify (ii), we use the Marian--Zhao theorem \cite[Theorem 4.3(i)]{LiYuZhang}: for any $z' = \{b_0,\dots,b_3\}\in \mathrm{Kum}_3(X)$, we have $z' \in O_z$ if and only if $c_2(z) = c_2(z')$; similarly, for $y\in \Km(X)^{[3]}$, we have $y \in O_{\rho(z)}$ if and only if $c_2(y) = c_2(\rho(z))$.
The Chern class of $\rho(z)$ is 
\[
c_2(\rho(z)) = \chow{\varpi(a_0 + a_1)} + \chow{\varpi(a_0 + a_2)} + \chow{\varpi(a_0 + a_3)}.
\] 
From \eqref{eq:alt-sum}, we see that $c_2(z) = c_2(z')$ if and only if
\begin{equation}\label{eq:alt-sum-c2}
   \sum_{i<j} \chow{a_i + a_j} = \sum_{i<j} \chow{b_i + b_j}.
 \end{equation}
Applying the isomorphism $\varpi_*$ gives 
\[
\sum_{i<j} \chow{\varpi(a_i + a_j)} = \sum_{i<j} \chow{\varpi(b_i + b_j)},
\]
i.e. $c_2(\rho(z)) = c_2(\rho(z'))$. Hence $\rho$ preserves the rational equivalence relation, establishing (ii).

Now assume $\dim O_z \ge 3 - i$. Then by the claim, $\dim O_{\rho(z)} \ge 3 - i$. As Voisin's filtration agrees with Shen--Yin--Zhao's filtration on $\CH_0(\km(X)^{[3]})$ in the strong sense (see \cite[Theorem 9]{Voisin15}), we obtain
\[
c_2(\rho(z)) = \sum_{i<j} \chow{\varpi(a_i + a_j)} \in \mathbf{S}_i(\Km(X)).
\] 
From the definition of $\mathbf{S}_i(X)$, it follows that
\[
\sum_{i<j} \chow{a_i + a_j} \in \mathbf{S}_i(X),
\]
which yields $c_2(z) \in \mathbf{S}_i(X)$. Hence $\mathbf{S}_\bullet \CH_0(K) \subseteq \mathbf{S}_\bullet^{\rm SYZ} \CH_0(K)$.

The reverse inclusion follows by a symmetric computation, using the same alternating sum relation in the opposite direction. Therefore
\[
\mathbf{S}_\bullet^{\rm SYZ} \CH_0(K) = \mathbf{S}_\bullet \CH_0(K). \qedhere
\]
\end{proof}

\begin{remark}
The preceding calculation is compatible with Lin’s splitting of Voisin’s filtration on generalized Kummer varieties \cite{Lin16}, which gives another way to view the equality in \Cref{thm:SYZ-sixfold}. Recall the following diagram:
\[
\begin{tikzcd}
& X_0^{n+1} \arrow[r] \arrow[d] & X^{n+1} \arrow[r, "\sum"] \arrow[d, "q"] & X \arrow[d] \\
K_n(X) \arrow[r] & K_{(n)}(X) \arrow[r] & X^{(n+1)} \arrow[r, "\sum"] & X
\end{tikzcd}
\]
Here $X_0^{n+1} := \ker\bigl( \sum: X^{n+1} \to X \bigr)$ is the kernel of the summation map, $K_{(n)}(X) = X_0^{n+1} / \mathfrak{S}_{n+1}$ is the singular generalized Kummer model. Since $K_n(X)$ is a crepant resolution of $K_{(n)}(X)$, there is an isomorphism
\[
\CH_0(K_n(X)) \xrightarrow{\sim} \CH_0(K_{(n)}(X)) = \CH_0(X_0^{n+1})^{\mathfrak{S}_{n+1}}.
\]
As the $\mathfrak{S}_{n+1}$-action on $\CH_0(X_0^{n+1})$ is compatible with the Beauville decomposition on $\CH_0(X_0^{n+1})$, there is an induced decomposition $\bigoplus_{0 \le s \le 2n} \CH_0(K_n(X))_s$, giving a split filtration. Lin proves that Voisin's filtration coincides with this split filtration, namely (cf. \cite[Theorem 1.5]{Lin16})
\[
\mathbf{S}_i \CH_0(K_n(X)) = \bigoplus_{2s \le 2n - 2i} \CH_0(K_n(X))_{2s}.
\] 
\end{remark}

\subsection{Floccari's construction for sixfolds}
In this subsection, we use an explicit construction to prove \Cref{conj:anti-symplectic} for twisted modular sixfolds. In fact, the proof strategy is valid for all hyperk\"ahler sixfolds of $\mathrm{Kum}_3$-type for which $G_0$ acts trivially on $\CH_0$. This construction is independent of \Cref{conj:SYZ}.

Recall that a key feature established in \cite{Fl24} associates a $\mathrm{K3}^{[3]}$-type variety to any $\mathrm{Kum}_3$-type sixfold.
To be precise, let $K$ be a smooth projective hyperkähler sixfold of $\mathrm{Kum}_3$-type, and let
\[
G_0 \cong (\mathbb{Z}/2\mathbb{Z})^5 \subset \operatorname{Aut}_0(K)
\]
be the finite subgroup of $\Aut_0(K)$, generated by the automorphisms whose fixed locus contains a 4-dimensional
component \cite[Lemma 2.4]{Fl24}. 
It consists of all involutions in $\operatorname{Aut}_0(K)$, and is a normal subgroup of $\operatorname{Bir}(K)$ since any conjugation maps an involution to an involution. The quotient $K/G_0$ is singular, but it admits a crepant resolution
\[
\pi: Y_K \longrightarrow K/G_0,
\]
where $Y_K$ is a hyperkähler manifold of $\mathrm{K3}^{[3]}$-type.
Moreover, when $K$ is projective, there exists a uniquely determined K3 surface $S_K$ such that $Y_K$ is birational to a moduli space of stable sheaves on $S_K$.

\begin{remark}
When $K = K_3(X)$, the birational correspondence constructed above is compatible with the  map $\rho: K_3(X) \dashrightarrow \Km(X)^{[3]}$ given in \Cref{eq:kummer-sym3-map}. Floccari introduces a stratification $(K/G_0)^j$ on $K/G_0$, whose local model is given by (cf. \cite[Proposition 4.2]{Fl24})
\[
\bigl(\operatorname{Bl}_0(\mathbb C^2/\iota)\bigr)^j \times (\mathbb C^2)^{3-j}, \quad 0 \le j \le 3.
\]
This is precisely the local model obtained from the product of the minimal resolution $\Km(X) \to X/\pm1$. Thus the birational correspondence
\[
\rho: Y_K \dashrightarrow \Km(X)^{[3]}
\]
is locally identified over each stratum of $(X/\pm1)^{(3)}$. These local identifications are compatible on overlaps and hence glue to a global isomorphism over $(X/\pm1)^{(3)}$, recovering the map $\rho$ from \Cref{eq:kummer-sym3-map}.
\end{remark}

The construction above gives a sequence of maps
\[
\CH_0(Y_K) \xlongrightarrow[\sim]{\pi_*} \CH_0(K/G_0) \xlongrightarrow[\sim]{q^*} \CH_0(K)^{G_0} \longleftarrow \CH_0(K).
\]
Here $q^*$ identifies $\CH_0(K/G_0)$ with the $G_0$-invariant part $\CH_0(K)^{G_0}$. 

The following theorem generalizes the equality $\dim O_z = \dim O_{\rho(z)}$ established in \Cref{thm:SYZ-sixfold} from $K = K_3(X)$ to any $\mathrm{Kum}_3$-type hyperkähler manifold on which $G_0$ acts trivially on the Chow group.

\begin{theorem}\label{thm:rational-orbit-sixfolds}
Let $K$ be a $\mathrm{Kum}_3$-type hyperk\"ahler manifold. If $G_0$ acts trivially on $\CH_0(K)$, then we have an isomorphism
\[
\CH_0(Y_K) \xrightarrow{\cong} \CH_0(K)
\]
which induces an isomorphism between Voisin's filtrations
\[
\mathbf{S}_\bullet \CH_0(K) \cong \mathbf{S}_\bullet \CH_0(Y_K).
\]
\end{theorem}

\begin{proof}
By assumption, $G_0$ acts trivially on $\CH_0(K)$; hence $\CH_0(K)^{G_0} \cong \CH_0(K)$.  
Using the identification $\CH_0(Y_K) \cong \CH_0(K/G_0) \cong \CH_0(K)^{G_0}$, we obtain an isomorphism
\[
\chi : \CH_0(Y_K) \longrightarrow \CH_0(K).
\]

We now show that $\chi$ respects Voisin's filtration, i.e. $\chi\bigl(\mathbf{S}_i\CH_0(Y_K)\bigr) = \mathbf{S}_i\CH_0(K)$ for all $i$.
Let $q : K \to K/G_0$ be the quotient map and $\pi : Y_K \to K/G_0$ be the crepant resolution.  
Then $\chi$ can be written as the composition
\[
\CH_0(Y_K) \xlongrightarrow{\pi_*} \CH_0(K/G_0) \xlongrightarrow{q^*} \CH_0(K)^{G_0} \hookrightarrow \CH_0(K).
\]
\begin{description}[font=\it]
    \item[$\bullet$ $\chi$ preserves the filtration] 
    Take any $z \in Y_K$ with $\dim O_z \geq 3-i$. 
    Choose an irreducible component $Z \subseteq O_z$ of maximal dimension. By \cite[Corollary 5.4]{CLZZ:2024}, such components are Zariski dense in $O_z$; hence we may assume $Z$ is not contained in the exceptional locus of $\pi$.  
    Consider $q^{-1}(\pi(Z)) \subseteq K$. Since $\pi$ is birational and $Z$ is constant-cycle, $\pi(Z)$ is also constant-cycle in $K/G_0$; $q$ being finite also preserves the constant-cycle property. Moreover, $\dim q^{-1}(\pi(Z)) = \dim Z \ge 3-i$.  
    The class $\chi(\chow{z})$ is therefore represented by a cycle supported on $q^{-1}(\pi(Z))$ (it is the average of the lifts of $\pi(Z)$ under the $G_0$-action). Consequently, $\chi(\chow{z}) \in \mathbf{S}_i \CH_0(K)$.

    \item[$\bullet$ $\chi^{-1}$ preserves the filtration]
    Conversely, take any point $x \in K$ with $\dim O_x \ge 3-i$.  
     Let $W \subseteq O_x$ be an irreducible component of maximal dimension.  As $q$ is finite and surjective, $q(W)$ is an irreducible closed subset of $K/G_0$ of the same dimension.  Indeed, the rational equivalence orbit $O_{q(x)}$ is exactly the image of $O_x$ via $q$.  Note that the maximal irreducible components of $O_{q(x)}$ are also Zariski dense; we can choose $W$ generic so that $q(W)$ is not contained in the singular locus of $K/G_0$.  
    The strict transform of $q(W)$ in $Y_K$ under $\pi^{-1}$ has dimension $\dim W \ge 3-i$ and consists of constant-cycle points.  
    By construction, $\chi^{-1}(\chow{x})$ is represented by a point on this strict transform, hence $\chi^{-1}(\chow{x}) \in \mathbf{S}_i\CH_0(Y_K)$. \qedhere
\end{description}
\end{proof}

\begin{corollary} \label{cor:twisted-modular-sixfold}
For a twisted modular sixfold $K = K_\sigma(X,\alpha,\mathbf{v})$, the hypotheses of \Cref{thm:rational-orbit-sixfolds} are satisfied. Consequently,
\[
\mathbf{S}_\bullet \CH_0(K) \cong \mathbf{S}_\bullet \CH_0(Y_K).
\]
\end{corollary}

\begin{proof}
The generators of $G_0$ are described in \Cref{subsection:monodromy-and-numerical-trivial}. They act trivially on $c_2(E)$ for every $E \in K$, hence by \Cref{thm:Marian-Zhao} they act trivially on $\CH_0(K)$. Thus $G_0$ acts trivially on $\CH_0(K)$, and \Cref{thm:rational-orbit-sixfolds} applies. 
\end{proof}

\begin{theorem} \label{thm:modular-sixfold}
   If $K = K_\sigma(X,\alpha,\mathbf{v})$ is a twisted modular sixfold, then \Cref{conj:anti-symplectic} holds for anti-symplectic birational automorphisms.
\end{theorem}

\begin{proof}
Let $f \in \Bir(K)$ be anti-symplectic. Since $G_0$ is a normal subgroup of $\Bir(K)$ and $f$ normalizes $G_0$, $f$ descends uniquely to $\overline{f} \in \Bir(K/G_0)$, which lifts to $\widetilde{f} \in \Bir(Y_K)$ via the crepant resolution.

Now $Y_K$ is a hyperkähler manifold of $\mathrm{K3}^{[3]}$-type with Picard number $\ge 16$, and $\widetilde{f}$ is anti-symplectic. By the main result of \cite{LiYuZhang}, $\widetilde{f}$ acts on $\CH_0(Y_K)$ as $(-1)^i \operatorname{id}$ on each $\mathbf{S}_i\CH_0(Y_K)$.

From \Cref{thm:rational-orbit-sixfolds}, the rational map $K\dashrightarrow Y_K$ induces an isomorphism $\CH_0(Y_K) \cong \CH_0(K)$ which identifies Voisin's filtrations. Hence $f$ acts on $\CH_0(K)$ as $(-1)^i \operatorname{id}$ on each $\mathbf{S}_i\CH_0(K)$. This completes the proof.
\end{proof}

\printbibliography
\end{document}